\documentclass[12pt]{amsart}
\usepackage{amsmath, amssymb}
\usepackage{amsfonts}
\usepackage{amscd}
\usepackage{tikz}
\usepackage{amsthm}

\usetikzlibrary{arrows}

\allowdisplaybreaks

\newtheorem{theorem}{Theorem}[section]
\newtheorem{lemma}[theorem]{Lemma}
\newtheorem{proposition}[theorem]{Proposition}

\newtheorem{corollary}[theorem]{Corollary}
\newtheorem*{theorem*}{Theorem}
\theoremstyle{definition}
\newtheorem{notation}[theorem]{Notation}
\newtheorem{remark}[theorem]{Remark}
\newtheorem{definition}[theorem]{Definition}
\newtheorem{example}[theorem]{Example}

\def\edge(#1,#2){(e_{#1},#2)}
\newcommand{\Xset}[1]{X(e_{#1})}

\numberwithin{equation}{section}

\newcommand{\N}{\mathbb{N}}

\newcommand{\C}{\mathbb{C}}
\newcommand{\T}{\mathbb{T}}

\newcommand{\G}{\mathcal{G}}
\newcommand{\cH}{\mathcal{H}}
\newcommand{\A}{\mathfrak{A}}

\newcommand{\wt}[1]{\widetilde{#1}}

\newcommand{\Pow}{\mathcal{P}}
\newcommand{\rg}{\textnormal{rg}}
\newcommand{\fin}{\textnormal{fin}}

\DeclareMathOperator{\spa}{span}
\DeclareMathOperator{\cspa}{\overline{span}}

\newcommand{\lcm}{\operatorname{lcm}}
\newcommand{\rom}{\renewcommand{\labelenumi}{{\rm (\roman{enumi})}}%
\renewcommand{\itemsep}{0pt}}
\newcommand{\Ca}{$C^*$-al\-ge\-bra }
\newcommand{\CA}{$C^*$-al\-ge\-bra}
\newcommand{\Csa}{$C^*$-sub\-al\-ge\-bra }
\newcommand{\Csas}{$C^*$-sub\-al\-ge\-bras }
\newcommand{\shom}{$*$-ho\-mo\-mor\-phism }

\begin{document}
\title[Graph algebras, Exel-Laca algebras, and ultragraph algebras]
{Graph algebras, Exel-Laca algebras, and ultragraph algebras coincide up to Morita equivalence}

\author{Takeshi Katsura}
\address{Takeshi Katsura, Department of Mathematics\\ Keio University\\
Yokohama, 223-8522\\ JAPAN}
\email{katsura@math.keio.ac.jp}

\author{Paul S. Muhly}
\address{Paul S.~Muhly, Department of Mathematics\\ University of Iowa\\
Iowa City\\ IA 52242-1419\\ USA}
\email{pmuhly@math.uiowa.edu}

\author{Aidan Sims}
\address{Aidan Sims, School of Mathematics and Applied Statistics\\
University of Wollongong\\
NSW 2522\\
AUSTRALIA} \email{asims@uow.edu.au}

\author{Mark Tomforde}
\address{Mark Tomforde \\ Department of Mathematics\\ University of Houston\\
Houston \\ TX 77204-3008\\ USA}
\email{tomforde@math.uh.edu}

\thanks{
The first author was supported by the Fields Institute, the
second author was supported by NSF Grant DMS-0070405, the third
author was supported by the Australian Research Council, and the fourth author was supported by an internal grant from the University of Houston Mathematics Department.}

\date{September 1, 2008; minor revisions December 7, 2008}


\begin{abstract}
We prove that the classes of graph algebras, Exel-Laca algebras, and ultragraph algebras coincide up to Morita
equivalence. This result answers the long-standing open question of whether every Exel-Laca algebra is Morita equivalent to a graph algebra.  Given an ultragraph $\G$ we
construct a directed graph $E$ such that $C^*(\G)$ is
isomorphic to a full corner of $C^*(E)$.  As applications, we characterize real rank zero for ultragraph algebras and describe quotients of ultragraph algebras by gauge-invariant ideals.
\end{abstract}

\maketitle

\section{Introduction}

In 1980 Cuntz and Krieger introduced a class of $C^*$-algebras
associated to finite matrices \cite{CK}. Specifically, if $A$ is an $n
\times n$ $\{0,1\}$-matrix with no zero rows, then the
Cuntz-Krieger algebra $\mathcal{O}_A$ is generated by partial
isometries $S_1, \dots , S_n$ such that $S_i^*S_i = \sum_{A(i,j) =1} S_jS_j^*$.
Shortly thereafter Enomoto, Fujii, and Watatani \cite{EW, FW, Wat} observed
that Cuntz and Krieger's algebras could be
described very naturally in terms of finite directed graphs. Given a
finite directed graph $E$ in which every vertex emits at least
one edge, the corresponding $C^*$-algebra $C^*(E)$ is generated
by mutually orthogonal projections $P_v$ associated to the vertices and partial
isometries $S_e$ associated to the edges such that $S_e^*S_e = P_{r(e)}$ and $P_v = \sum_{s(e)=v} S_eS_e^*$, where $r(e)$ and $s(e)$ denote the range and source of an edge $e$.

Attempting to generalize the theory of Cuntz-Krieger algebras
to countably infinite generating sets resulted in two very
prominent classes of $C^*$-algebras: graph $C^*$-algebras and
Exel-Laca algebras. To motivate our results we briefly describe
each of these classes.  The key issue for both generalizations
is that infinite sums of projections, which a naive approach
would suggest, cannot converge in norm.

To generalize graph $C^*$-algebras to infinite graphs, the key
modification is to require the relation $P_v = \sum_{s(e)=v}
S_eS_e^*$ to hold only when the sum is finite and nonempty.
This theory has been explored extensively in many papers (see
\cite{KPR, KPRR, BPRS2000, FLR} for seminal results, and
\cite{Raeburn2005} for a survey). Graph $C^*$-algebras include
many $C^*$-algebras besides the Cuntz-Krieger algebras; in
particular, graph $C^*$-algebras include the Toeplitz algebra,
continuous functions on the circle, all finite-dimensional
$C^*$-algebras, many AF-algebras, many purely infinite simple
$C^*$-algebras, and many Type I $C^*$-algebras. From a
representation-theoretic point of view, the class of graph
$C^*$-algebras is broader still: every AF-algebra is Morita
equivalent to a graph $C^*$-algebra, and any Kirchberg algebra
with free $K_1$-group is Morita equivalent to a graph
$C^*$-algebra.

The approach taken for Exel-Laca algebras is to allow the
matrix $A$ to be infinite. Here rows containing
infinitely many nonzero entries lead to an infinite sum of
projections, which does not give a sensible relation. However, Exel and Laca observed that even when
rows of the matrix contain infinitely many nonzero entries, formal combinations of the Cuntz-Krieger relations can result in relations of the form $\prod_{i \in X} S^*_i S_i \prod_{j \in Y} (1 - S^*_j S_j) = \sum_{k \in Z} S_kS_k^*$, where $X$, $Y$, and $Z$ are all finite.  It is precisely these finite relations that are imposed in the definition of the Exel-Laca algebra. As with the graph
$C^*$-algebras, the Exel-Laca algebras include many classes of
$C^*$-algebras in addition to the Cuntz-Krieger algebras, and
numerous authors have studied their structure \cite{EL, EL2,
RS, Szy4}.

Without too much effort, one can show that neither the class of
graph $C^*$-algebras nor the class of Exel-Laca algebras is a
subclass of the other.  Specifically, there exist graph
$C^*$-algebras that are not isomorphic to any Exel-Laca algebra
\cite[Proposition A.16.2]{Tom-thesis}, and there exist
Exel-Laca algebras that are not isomorphic to any graph
$C^*$-algebra \cite[Example~4.2 and Remark~4.4]{RS}.  This
shows, in particular, that there is merit in studying both
classes and that results for one class are not special cases of
results for the other.  It also begs the question, ``How different
are the classes of graph $C^*$-algebras and Exel-Laca
algebras?"  Although each contains different isomorphism
classes of $C^*$-algebras, a natural follow-up question is to
ask about Morita equivalence.  Specifically,
\vskip2ex
\noindent \textsc{Question 1:} Is every graph $C^*$-algebra Morita
    equivalent to an Exel-Laca algebra?
\vskip2ex
\noindent \textsc{Question 2:} Is every Exel-Laca algebra Morita
    equivalent to a graph $C^*$-algebra?
\vskip2ex

While the question of isomorphism is easy to sort out, the
Morita equivalence questions posed above are much more
difficult.  Question~1 was answered in the affirmative by
Fowler, Laca, and Raeburn in \cite{FLR}.  In particular, if
$C^*(E)$ is a graph $C^*$-algebra, then one may form a graph
$\tilde{E}$ with no sinks or sources, by adding tails to the
sinks of $E$ and heads to the sources of $E$.  A standard
argument shows that $C^*(\tilde{E})$ is Morita equivalent to
$C^*(E)$ (see \cite[Lemma~1.2]{BPRS2000}, for example), and
Fowler, Laca, and Raeburn proved that the $C^*$-algebra of a
graph with no sinks and no sources is isomorphic to an
Exel-Laca algebra \cite[Theorem~10]{FLR}.

On the other hand, Question~2 has remained an open problem for
nearly a decade. The various invariants
calculated for graph $C^*$-algebras and Exel-Laca algebras have
not been able to discern any Exel-Laca algebras that are not
Morita equivalent to a graph $C^*$-algebra.  For example, the
attainable $K$-theories for both classes are the same: all
countable free abelian groups arise as $K_1$-groups together
with all countable abelian groups as $K_0$-groups (see
\cite{Szy2} and \cite{EL2}). Nevertheless, up to this
point there has been no method for constructing a graph $E$
from a matrix $A$ so that $C^*(E)$ is Morita equivalent to
$\mathcal{O}_A$.

In this paper we provide an affirmative answer to Question~2
using a generalization of a graph known as an ultragraph.
Ultragraphs and the associated $C^*$-algebras were
introduced by the fourth author to unify the
study of graph $C^*$-algebras and Exel-Laca algebras \cite{Tom,
Tom2}.  An ultragraph is a generalization of a graph in which
the range of an edge is a (possibly infinite) set of vertices,
rather than a single vertex. The ultragraph $C^*$-algebra is
then determined by generators satisfying relations very similar
to those for graph $C^*$-algebras (see Section~\ref{sec:prelims}). The fourth author has shown
that every graph algebra is isomorphic to an ultragraph algebra
\cite[Proposition~3.1]{Tom}, every Exel-Laca algebra is
isomorphic to an ultragraph algebra \cite[Theorem~4.5,
Remark~4.6]{Tom}, and moreover there are ultragraph algebras
that are not isomorphic to any graph algebra and are not
isomorphic to any Exel-Laca algebra \cite[\S 5]{Tom2}.  Thus
the class of ultragraph algebras is strictly larger than the union of the two classes of Exel-Laca algebras and of graph algebras.

In addition to providing a framework for studying graph
algebras and Exel-Laca algebras simultaneously, ultragraph
algebras also give an alternate viewpoint for studying
Exel-Laca algebras.   In particular, if $A$ is a (possibly
infinite) $\{0,1\}$-matrix, and if we let $\G$ be ultragraph
with edge matrix $A$, then the Exel-Laca algebra
$\mathcal{O}_A$ is isomorphic to the ultragraph algebra
$C^*(\G)$.  In much of the seminal work done on Exel-Laca
algebras \cite{EL, EL2, Szy4}, the structure of the
$C^*$-algebra $\mathcal{O}_A$ is related to properties of the
infinite matrix $A$ (see \cite[\S 2--\S 10]{EL} and also
\cite[Definition~4.1 and Theorem~4.5]{EL2}) as well as
properties of an associated graph $\textrm{Gr} (A)$ with edge
matrix $A$ (see \cite[Definition~10.5, Theorem~13.1,
Theorem~14.1, and Theorem~16.2]{EL} and
\cite[Theorem~8]{Szy4}).  Unfortunately, these correspondences
are often of limited use since the properties of the matrix can
be difficult to visualize, and the graph $\textrm{Gr} (A)$ does
not entirely reflect the structure of the Exel-Laca algebra
$\mathcal{O}_A$ (see \cite[Example~3.14 and
Example~3.15]{Tom2}).  Another approach is to represent
properties of the Exel-Laca algebra in terms of the ultragraph
$\G_A$ \cite{Tom, Tom2}. This is a useful technique because it
gives an additional way to look at properties of Exel-Laca
algebras, the ultragraph $\G_A$ reflects much of the fine
structure of the Exel-Laca algebra $\mathcal{O}_A$, and
furthermore the interplay between the ultragraph and the
associated $C^*$-algebra has a visual nature similar to what
occurs with graphs and graph $C^*$-algebras.

In this paper we prove that the classes of graph algebras,
Exel-Laca algebras, and ultragraph algebras coincide up to
Morita equivalence.  This provides an affirmative answer to
Question~2 above, and additionally shows that no new Morita
equivalence classes are obtained in the strictly larger class
of ultragraph algebras.  Given an ultragraph $\G$ we build a graph $E$ with the property that $C^*(\G)$ is isomorphic to a full corner of $C^*(E)$.   Combined with other known results, this shows our three classes of $C^*$-algebras coincide up to
Morita equivalence.  Since our construction is concrete, we are also able to use graph algebra results to
analyze the structure of ultragraph algebras.  In particular,
we characterize real rank zero for ultragraph
algebras, and describe the quotients of ultragraph algebras by
gauge-invariant ideals.  Of course, these structure results
also give corresponding results for Exel-Laca algebras as
special cases, and these results are new as well.  In addition,
our construction implicitly gives a method for taking a
$\{0,1\}$-matrix $A$ and forming a graph $E$ with the property
that the Exel-Laca algebra is isomorphic to a full corner of
the graph $C^*$-algebra $C^*(E)$ (see
Remark~\ref{rmk:E-L-full-corn}).

It is interesting to note that we have used ultragraphs to answer Question~2, even though
this question is intrinsically only about graph $C^*$-algebras and Exel-Laca algebras.
Indeed it is difficult to see how to answer Question~2 without at least implicit recourse
to ultragraphs.  This provides additional evidence that
ultragraphs are a useful and natural tool for exploring the
relationship between Exel-Laca algebras and graph $C^*$-algebras.

This paper is organized as follows.  After some preliminaries
in Section~\ref{sec:prelims}, we describe our construction in
Section~\ref{graph-sec} and explain how to build a graph $E$
from an ultragraph $\G$.  Since this construction is somewhat
involved, we also provide a detailed example for a particular
ultragraph at the end of this section.  In
Section~\ref{Path-sec} we analyze the path structure of the
graph constructed by our method.  In Section~\ref{corner-sec}
we show that there is an isomorphism $\phi$ from the ultragraph
algebra $C^*(\G)$ to a full corner $P C^*(E) P$ of $C^*(E)$,
and we use this result to show that an ultragraph algebra $C^*(\G)$ has
real rank zero if and only if $\G$ satisfies Condition~(K).  In
Section~\ref{ideal-sec}, we prove that the induced bijection $I
\mapsto C^*(E) \phi(I) C^*(E)$ restricts to a bijection between gauge-invariant ideals of
$C^*(\G)$ and gauge-invariant ideals of $C^*(E)$. In
Section~\ref{quotient-sec}, we give a complete description of
the gauge-invariant ideal structure of ultragraph algebras
commenced in \cite{KMST}, by describing the quotient of an
ultragraph algebra by a gauge-invariant ideal.

\section{Preliminaries} \label{sec:prelims}

For a set $X$, let $\Pow(X)$ denote the collection of all
subsets of $X$.  We recall from \cite{Tom} the definitions of an ultragraph and
of a Cuntz-Krieger family for an ultragraph.

\begin{definition}(\cite[Definition~2.1]{Tom})
An \emph{ultragraph} $\G = (G^0, \G^1, r,s)$ consists of a
countable set of vertices $G^0$, a countable set of edges
$\G^1$, and functions $s\colon \G^1\rightarrow G^0$ and
$r\colon \G^1 \rightarrow \Pow(G^0)\setminus\{\emptyset\}$.
\end{definition}

The original definition of a Cuntz-Krieger family for an
ultragraph $\G$ appears as \cite[Definition~2.7]{Tom}. However,
for our purposes, it will be more convenient to work with the
Exel-Laca $\G$-families of \cite[Definition~3.3]{KMST}.  To give this definition, we first recall that for finite subsets $\lambda$ and $\mu$ of $\G^1$, we define
\[
r(\lambda,\mu) := \bigcap_{e \in \lambda} r(e) \setminus \bigcup_{f \in \mu} r(f)\in \Pow(G^0).
\]

\begin{definition}\label{dfn:Cond(EL)}
Let $\G =(G^0, \G^1, r,s)$ be an ultragraph. A collection of
projections $\{p_v : v\in G^0\}$ and $\{q_e : e \in \G^1\}$ is
said to satisfy {\em Condition (EL)} if the following hold:
\begin{enumerate}
\item the elements of $\{p_v : v\in G^0\}$ are pairwise
   orthogonal,
\item the elements of $\{q_e : e \in \G^1\}$ pairwise
   commute,
\item $p_v q_e= p_v$ if $v\in r(e)$, and $p_v q_e=0$ if $v
   \notin r(e)$,
\item $\prod_{e\in \lambda}q_e\prod_{f\in \mu}(1-q_f)
   =\sum_{v\in r(\lambda,\mu)}p_v$ for all finite subsets
   $\lambda,\mu$ of $\G^1$ such that $\lambda\cap
   \mu=\emptyset$, $\lambda\neq\emptyset$ and
   $r(\lambda,\mu)$ is finite.
\end{enumerate}
\end{definition}

Given an ultragraph $\G$, we write $G^0_{\rg}$ for the set $\{v
\in G^0 : s^{-1}(v)\text{ is finite and nonempty}\}$ of
\emph{regular vertices} of $\G$.

\begin{definition}\label{dfn:EL-G-fam}
For an ultragraph $\G = (G^0, \G^1, r,s)$, an \emph{Exel-Laca
$\G$-family} is a collection of projections $\{p_v : v\in
G^0\}$ and partial isometries $\{s_e : e \in \G^1\}$ with
mutually orthogonal final projections for which
\begin{enumerate}
\item the collection $\{p_v : v\in G^0\} \cup \{s_e^*s_e :
   e\in \G^1\}$ satisfies Condition~(EL),
\item $s_es_e^* \leq p_{s(e)}$ for all $e \in \G^1$,
\item $p_v = \sum_{s(e) = v} s_es_e^*$ for $v\in
   G^0_{\rg}$.
\end{enumerate}
\end{definition}

Our use of the notation $\G^0$ in what follows will also be in
keeping with \cite{KMST} rather than with \cite{Tom, Tom2}.

By a \emph{lattice} in $\Pow(X)$, we mean a collection of
subsets of $X$ which is closed under finite intersections and
unions.  By an \emph{algebra} in $\Pow(X)$, we mean a lattice
in $\Pow (X)$ which is closed under taking relative
complements. As in \cite{KMST}, we denote by $\G^0$ the
smallest algebra in $\Pow(G^0)$ which contains both $\{\{v\} :
v \in G^0\}$ and $\{r(e) : e \in \G^1\}$ (by contrast, in
\cite{Tom, Tom2}, $\G^0$ denotes the smallest lattice in
$\Pow(G^0)$ containing these sets). A \emph{representation} of
an algebra $\mathfrak{A}$ in a $C^*$-algebra $B$ is a
collection $\{p_A : a \in \mathfrak{A}\}$ of mutually commuting
projections in $B$ such that $p_{A \cap B} = p_A p_B$, $p_{A
\cup B} = p_A + p_B - p_{A \cap B}$ and $p_{A \setminus B} =
p_A - p_{A \cap B}$ for all $A,B \in \mathfrak{A}$.

Given an Exel-Laca $\G$-family $\{p_v : v\in G^0\}$, $\{s_e : e
\in \G^1\}$ in a $C^*$-algebra $B$,
\cite[Proposition~3.4]{KMST} shows that there is a unique
representation $\{p_A : A \in \G^0\}$ of $\G^0$ such that
$p_{r(e)} = s^*_e s_e$ for all $e \in E^0$, and $p_{\{v\}} = p_v$
for all $v \in G^0$. In particular, given an Exel-Laca
$\G$-family $\{p_v, s_e\}$, we will without comment denote the
resulting representation of $\G^0$ by $\{p_A : A \in \G^0\}$;
in particular, $p_{r(e)}$ denotes $s^*_e s_e$, and
$p_{\{v\}}$ and $p_v$ are one and the same.

\section{A directed graph constructed from an ultragraph} \label{graph-sec}

The purpose of this section is to construct
a graph $E=(E^0,E^1,r_E,s_E)$
from an ultragraph $\G = (G^0, \G^1, r,s)$.  Our construction involves a choice of a listing of $\G^1$ and of a function $\sigma$ with certain properties described in Lemma~3.7; in particular, different choices of listings and of sigma will yield different graphs.
In Section~\ref{corner-sec},
we will prove that
the ultragraph algebra $C^*(\G)$
is isomorphic to a full corner of
the graph algebra $C^*(E)$, regardless of the choices made.

\begin{notation}
Fix $n\in\N = \{1,2,\ldots\}$, and $\omega\in\{0,1\}^n$. For
$i=1,2,\ldots,n$, we denote by $\omega_i \in \{0,1\}$ the
$i$\textsuperscript{th} coordinate of $\omega$, and we denote $n$ by $|\omega|$.

We express $\omega$ as $(\omega_1, \omega_2, \dots, \omega_n)$.
We define $(\omega, 0), (\omega, 1) \in
\{0,1\}^{n+1}$ by $(\omega, 0) := (\omega_1, \omega_2, \dots,
\omega_n, 0)$ and $(\omega, 1) := (\omega_1, \omega_2, \dots,
\omega_n, 1)$. For $m\in\N$ with $m\leq n$, we define
$\omega|_m\in\{0,1\}^m$ by $\omega|_m=(\omega_1, \omega_2,
\dots, \omega_m)$. The elements $(0,0,\dots, 0,0)$ and
$(0,0,\dots, 0, 1)$ in $\{0,1\}^n$ are denoted by $0^n$ and
$(0^{n-1},1)$.
\end{notation}

Let $\G$ be an ultragraph $(G^0, \G^1, r,s)$. Fix an ordering on
$\G^1 = \{e_1, e_2, e_3, \ldots\}$. (This list may be finite or
countably infinite.) Using the same notation as established in
\cite[Section~2]{KMST}, we define
$r(\omega):=\bigcap_{\omega_i=1}r(e_i)\setminus
\bigcup_{\omega_j=0}r(e_j)\subset G^0$ for
$\omega\in\{0,1\}^n\setminus \{0^n\}$, and $\Delta_n
:=\big\{\omega\in \{0,1\}^n\setminus \{0^n\}:
|r(\omega)|=\infty\big\}.$ If $\omega \in \{0,1\}^n$ and $i \in
\{0,1\}$ so that $\omega' := (\omega,i) \in \{0,1\}^{n+1}$, we
somewhat inaccurately write $r(\omega,i)$, rather than
$r((\omega,i))$,  for $r(\omega')$.

\begin{definition}\label{dfn:Delta0}
We define $\Delta:=\bigsqcup^\infty_{n=1}\Delta_n$. So for
$\omega\in \Delta$, we have $\omega \in \Delta_{|\omega|}$.
\end{definition}

\begin{remark}\label{rmk:omegai}
Since $r(\omega) = r(\omega, 0) \sqcup r(\omega, 1)$ for each
$\omega \in \Delta$, an element $\omega\in
\{0,1\}^{n}\setminus\{0^n\}$ is in $\Delta$ if and only if at
least one of the elements $(\omega, 0)$ and $(\omega, 1)$ is in
$\Delta$.
\end{remark}

\begin{definition}
We define $\Gamma_0 :=\{(0^n, 1) : n \ge 0, |r(0^n,1)| =
\infty\} \subset \Delta$, and $\Gamma_+ := \Delta \setminus
\Gamma_0$.
\end{definition}

We point out that $(0^n, 1)$ means $(1)$ when $n=0$.
Also, by Remark~\ref{rmk:omegai}, if  $n \in \N$, $\omega \in \Delta_{n+1}$, and
$\omega|_n \not= 0^n$, then $\omega|_n \in \Delta_n$.
Since $\omega \in \Delta_{n+1}$ satisfies $\omega|_n = 0^n$ if and only if
$\omega = (0^n,1) \in \Gamma_0$, it follows that $\Gamma_+ =
\{ \omega \in \Delta: |\omega| > 1 \text{ and } \omega|_{|\omega|-1} \in \Delta \}$.

\begin{definition}\label{dfn:X}
Let $W_+ := \bigcup_{\omega\in \Delta}r(\omega)\subset G^0$,
and $W_0 := G^0 \setminus W_+$.
\end{definition}

\begin{lemma}\label{lem:W_+=disjoint union}
We have $W_+ = \bigsqcup_{\omega \in
\Gamma_0} r(\omega)$.
\end{lemma}
\begin{proof}
For $\omega\in \Delta$,  let $m(\omega) := \min\{k : \omega_k = 1\}$.
Then $\omega|_{m(\omega)} \in \Gamma_0$ and $r(\omega) \subset r(\omega|_{m(\omega)})$. Thus $W_+ = \bigcup_{\omega \in
\Gamma_0}r(\omega)$. Finally, the sets $r(\omega)$ and
$r(\omega')$ are disjoint for distinct $\omega, \omega' \in
\Gamma_0$ by definition.
\end{proof}

\begin{lemma}\label{lem:W_+ -> Delta^0}
There exists a function $\sigma\colon W_+\to \Delta$ such that
$v\in r(\sigma(v))$ for each $v\in W_+$, and such that
$\sigma^{-1}(\omega)$ is finite (possibly empty) for each
$\omega\in \Delta$.
\end{lemma}
\begin{proof}
Let
\[
W_{\infty} := \{v \in W_+ : v \in r(\omega) \text{ for infinitely many $\omega
\in \Delta$}\}.
\]
For $v \in W_+ \setminus W_\infty$, we define $\sigma(v)$ to be
the element of $\Delta$ with $v \in r(\sigma(v))$ for which
$|\sigma(v)|$ is maximal.  Fix an ordering $\{v_1, v_2, v_3,
\dots\}$ of $W_{\infty}$, and fix $k \in \N$. The definition of
$W_\infty$ implies that the set $N_k := \{n \ge k : v_k \in
r(\omega)\text{ for some } \omega \in \Delta_n\}$ is infinite.
Let $n$ denote the minimal element of $N_k$. By
Lemma~\ref{lem:W_+=disjoint union}, there is a unique $\omega
\in \Delta_n$ such that $v_k \in r(\omega)$. We define
$\sigma(v_k) := \omega$.

By definition, we have $v\in r(\sigma(v))$ for each $v\in W_+$.
Fix $\omega\in \Delta$. We must show that $\sigma^{-1}(\omega)$
is finite. The set $\sigma^{-1}(\omega)\cap W_\infty$ is finite
because it is a subset of $\{v_1,v_2,\ldots,v_{|\omega|}\}
\subset W_\infty$. The set $\sigma^{-1}(\omega)\cap
(W_+\setminus W_\infty)$ is empty if both $(\omega, 0)$ and
$(\omega, 1)$ are in $\Delta$. Otherwise the set
$\sigma^{-1}(\omega)\cap (W_+\setminus W_\infty)$ coincides
with the finite set $r(\omega, i)$ for $i=0$ or $1$. Thus
$\sigma^{-1}(\omega)$ is finite.
\end{proof}

\begin{definition}
Fix a function $\sigma\colon W_+\to \Delta$ as in Lemma \ref{lem:W_+
-> Delta^0}. Extend $\sigma$ to a function $\sigma\colon G^0\to
\Delta\cup\{\emptyset\}$ by setting $\sigma(v)=\emptyset$ for
$v\in W_0$.
\end{definition}

We take the convention that $|\emptyset| = 0$ so that $v
\mapsto |\sigma(v)|$ is a function from $G^0$ to the
nonnegative integers. In particular, $v\in W_0$ if and only if
$|\sigma(v)|=0$, and $v\in W_+$ if and only if $|\sigma(v)|\geq
1$.

\begin{definition}\label{def:X_n}
For each $n\in\N$, we define a subset $\Xset{n}$ of $G^0\sqcup
\Delta$ by
\[
\Xset{n}:=\big\{v\in r(e_n):|\sigma(v)|<n\big\}\sqcup
\big\{\omega\in \Delta_n: \omega_n=1\big\}.
\]
\end{definition}

\begin{remark}\label{rmk:redundant e}
The occurrence of the symbol $e$ in the notation $\Xset{n}$ is
redundant; we might just as well label this set $X(n)$ or
$X_n$. However, we feel that it is helpful to give some hint
that the role of the $n$ in this notation is to pick out an
edge $e_n$ from our chosen listing of $\G^1$.
\end{remark}

\begin{lemma}\label{lem:Wnfin}
For each $n\in\N$, the set $\Xset{n}$ is nonempty and finite.
\end{lemma}
\begin{proof}
Fix $n \in \N$.
To see that $\Xset{n}$ is nonempty, suppose that $\{v\in
r(e_n):|\sigma(v)|<n\big\} = \emptyset$.
Since $r(e_n) \not=\emptyset$,
there exists $v \in r(e_n)$ such that $|\sigma(v)|\ge n$.
Set $\omega = \sigma(v)|_n$.
Since $v \in r(\sigma(v))\subset r(\omega)$,
we have $\omega_n = 1$ and $|r(\omega)|=\infty$,
so $\omega \in \Xset{n}$.
Thus $\Xset{n}$ is nonempty.

For $v \in r(e_n)$ with $|\sigma(v)| = 0$,
the element $\omega \in \{0,1\}^{n}$ satisfying $v \in r(\omega)$
is not in $\Delta_n$ by definition.
Hence the set $\{v\in r(e_n):|\sigma(v)| = 0\}$ is finite,
because it is a subset of the union of finitely many
finite sets $r(\omega)$
where $\omega \in \{0,1\}^n\setminus \Delta_n$.
Since $\sigma^{-1}(\omega)$ is finite for all
$\omega \in \Delta$ and since $\{\omega \in \Delta : |\omega| <
n\}$ is finite,
we have $|\{v\in r(e_n):0<|\sigma(v)|<n\}| < \infty$.
Since $\Delta_n$ is finite,
$\{\omega\in \Delta_n: \omega_n=1\big\}$ is also finite,
and thus $\Xset{n}$ is finite.
\end{proof}

\begin{definition}\label{dfn:E}
We define a graph $E=(E^0,E^1,r_E,s_E)$ as follows:
\begin{align*}
E^0 &:=G^0\sqcup \Delta, \\
E^1 &:=\{\overline{x} : x\in W_+\sqcup\Gamma_+\}
\sqcup\big\{\edge(n,x) : e_n \in \G^1,\ x\in \Xset{n}\big\}, \\
r_E(\overline{x})&:=x, \quad
r_E(\edge(n,x)) := x, \\
s_E(\overline{v})&:=\sigma(v), \quad
s_E(\overline{\omega}):=\omega|_{|\omega|-1},\quad
s_E(\edge(n,x)) := s(e_{n}).
\end{align*}
\end{definition}

\begin{remark}
Just as in Remark~\ref{rmk:redundant e}, the symbol $e$
here is redundant; we could simply have denoted the edge
$\edge(n,x)$ by $(n,x)$. We have chosen notation which is
suggestive of the fact that the $n$ is specifying an
element of $\G^1$ via our chosen listing.
\end{remark}

For the following proposition, recall $E^0_\rg$ denotes the set
$\{v \in E^0 : s_E^{-1}(v)$ is finite and nonempty$\}$ of
regular vertices of $E$. Also recall from
Section~\ref{sec:prelims} that $G^0_\rg$ denotes the set of
regular vertices of $\G$.

\begin{proposition}\label{prp:reg verts}
We have $E^0_{\rg}=G^0_{\rg}\sqcup\Delta$.
\end{proposition}

\begin{proof}
For $v\in G^0$, we have $s_E^{-1}(v) = \bigsqcup_{s(e_n) = v}
\{\edge(n,x): x\in \Xset{n}\}$. Hence Lemma \ref{lem:Wnfin}
implies that $s_E^{-1}(v)$ is nonempty and finite if and only
if $s^{-1}(v)\subset\G^1$ is nonempty and finite. Hence $v\in
E^0_{\rg}$ if and only if $v\in G^0_{\rg}$. Thus $E^0_{\rg}\cap
G^0=G^0_{\rg}$.

Fix $\omega\in \Delta$. By Remark~\ref{rmk:omegai}, we have
$(\omega, i) \in \Delta$ for at least one of  $i=0$ and $i=1$. Since
$s_E\big(\overline{(\omega, i)}\big)=\omega$, the set
$s_E^{-1}(\omega)$ is not empty. By definition of $\sigma$, the
set $s_E^{-1}(\omega)$ is finite. Hence $\omega\in E^0_{\rg}$.
Thus $E^0_{\rg}=G^0_{\rg}\sqcup\Delta$.
\end{proof}

\subsection*{An example of the graph constructed from an
ultragraph.}

We present an example of our construction, including an
illustration. This concrete example will, we hope, help the
reader to visualise the general construction and to keep track
of notation. The definition of the ultragraph $\G$ in this
example may seem involved, but has been chosen to illustrate as
many of the features of our construction as possible with a
diagram containing relatively few vertices.

\begin{example}\label{ex:illustration}
For the duration of this example we will omit the parentheses
and commas when describing elements of $\Delta$. For example,
the element $(0,0,1) \in \{0,1\}^3$ will be denoted $001$.

We define an ultragraph $\G = (G^0, \G^1, r, s)$ as follows.
Let $G^0 := \{v_n : n \in \N\}$ and $\G^1 := \{e_n : n \in
\N\}$. For each $n \in \N$, let $s(e_n) := v_n$. For $k\in \N$,
let
\begin{align*}
r(e_{2k-1})
&:= \{ v_m : (k+2) \text{ divides } m \},\\
r(e_{2k})
&:= \{ v_m : \text{$m \le k^2$ and $4$ does not divide $m$}\}.
\end{align*}

We will construct a graph $E$ from $\G$ as described in the earlier part of this section.  We have displayed $\G$ and $E$ in Figure~\ref{pic:G,E}.  To construct $E$, we must first choose a function $\sigma : W_+ \to \Delta$ as in Lemma~\ref{lem:W_+ -> Delta^0}. To do this, we first describe
$\Delta$ and $W_+$.

Fix $n \ge 1$ and $\omega \in
\{0,1\}^n \setminus \{0^n\}$. If $\omega_{2k} = 1$ for some $k
\in \N$, then $r(\omega) \subset \{v_1,v_2, \dots, v_{k^2}\}$
is finite, so $\omega \notin \Delta$. If $\omega_{2k-1} = 0$
and $\omega_{2l-1} = 1$ for $k,l\in \N$ with $(k+2) \mid
(l+2)$, then $r(\omega) = \emptyset$, so $\omega \notin
\Delta$. Indeed, we have $\omega \in \Delta$ if and only if:
\begin{enumerate}
\item $\omega_i = 1$ for some $i$;
\item\label{it:odd only} $\omega_i = 0$ for all even $i$;
  and
\item\label{it:divisibility} whenever $k$ satisfies
  $\omega_{2k-1} = 0$, we have $(k + 2) \nmid \lcm\{l + 2
  : \omega_{2l-1} = 1\}$.
\end{enumerate}
For example, note that
$\omega = 1010100$ and $\omega = 1010000$ are not in $\Delta$
(see Figure~\ref{pic:G,E}) because these have $\omega_1 =
\omega_3 = 1$, but $\omega_7 = 0$. Similarly, elements of the
form $0{*}{*}{*}{*}{*}1$ are missing.

To describe $\Gamma_0$, first observe that an element of the
form $0^n1$ can belong to $\Delta$ only if $n$ is even. For an
element $\omega$ of the form $0^{2n}1$, conditions (1)~and~(2)
above are trivially satisfied, and $\{l + 2
  : \omega_{2l-1} = 1\} = \{n + 3\}$, so
condition~(\ref{it:divisibility}) holds if and only if $k+2
\nmid n+3$ for all $k$ such that $2k-1 < 2n + 1$; that is, if
and only if $n+3$ is equal to $4$ or is an odd prime number.
Thus
\begin{align*}
\Gamma_0
&= \big\{0^{2(p-3)}1 : \text{$p=4$ or $p$ is an odd prime number}\big\} \\
&= \{1, 001, 00001, 0^81, 0^{16}1, 0^{20}1, 0^{28}1, \ldots\}.
\end{align*}

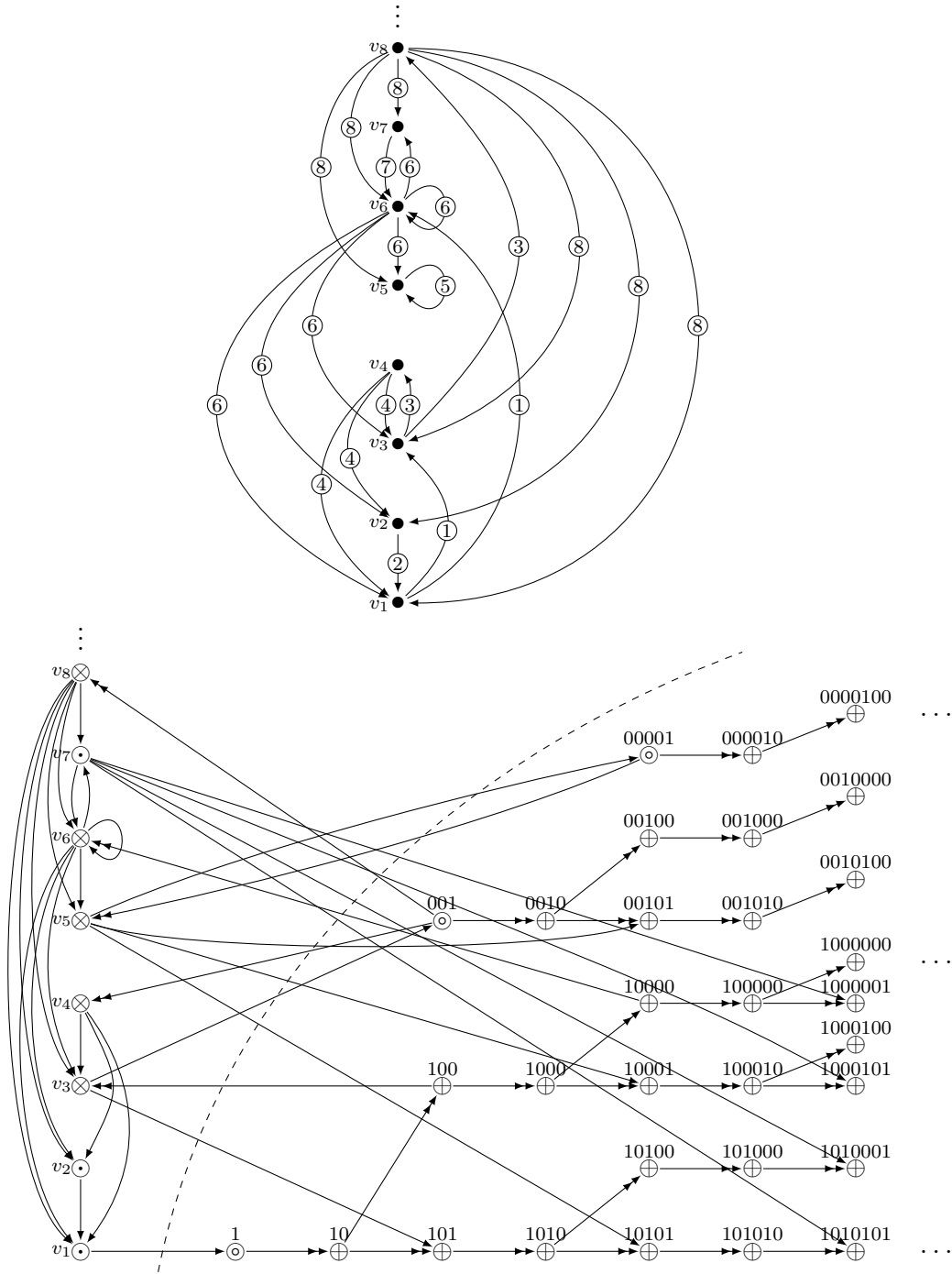
\begin{figure}[htp!]
\[
\begin{tikzpicture}[scale=1.15]
  \node[circle, inner sep=0.1pt] (v1) at (0, 0) {$\bullet$};
  \node[inner sep=-1pt, anchor=east] at (v1.195) {\tiny $v_1$};
  \node[circle, inner sep=0.1pt] (v2) at (0, 1) {$\bullet$};
  \node[inner sep=-1pt, anchor=east] at (v2.west) {\tiny $v_2$};
  \node[circle, inner sep=0.1pt] (v3) at (0,2) {$\bullet$};
  \node[inner sep=-1pt, anchor=east] at (v3.west) {\tiny $v_3$};
  \node[circle, inner sep=0.1pt] (v4) at (0,3) {$\bullet$};
  \node[inner sep=-1pt, anchor=east] at (v4.west) {\tiny $v_4$};
  \node[circle, inner sep=0.1pt] (v5) at (0,4) {$\bullet$};
  \node[inner sep=-1pt, anchor=east] at (v5.195) {\tiny $v_5$};
  \node[circle, inner sep=0.1pt] (v6) at (0,5) {$\bullet$};
  \node[inner sep=-1pt, anchor=east] at (v6.175) {\tiny $v_6$};
  \node[circle, inner sep=0.1pt] (v7) at (0, 6) {$\bullet$};
  \node[inner sep=-1pt, anchor=east] at (v7.west) {\tiny $v_7$};
  \node[circle, inner sep=0.1pt] (v8) at (0,7) {$\bullet$};
  \node[inner sep=-1pt, anchor=east] at (v8.west) {\tiny $v_8$};
  \node at (0,7.5) {\vdots};
  \begin{scope}[>=latex]
  \draw[->] (v1) .. controls +(0.8,0.75) and +(0.8,-0.75) .. (v3) node[circle, inner sep=0.5pt, pos=0.45, fill=white, draw] {\tiny 1};
  \draw[->] (v1) .. controls +(2,1) and +(2,-1) .. (v6) node[circle, inner sep=0.5pt, pos=0.5, fill=white, draw] {\tiny 1};
  \draw[->] (v2)--(v1) node[circle, inner sep=0.5pt, pos=0.5, fill=white, draw] {\tiny 2};
  \draw[->] (v3) .. controls +(0.175,0.25) and +(0.175, -0.25) .. (v4) node[circle, inner sep=0.5pt, pos=0.5, fill=white, draw] {\tiny 3};
  \draw[->] (v3) .. controls +(2,2) and +(2,-2) .. (v8) node[circle, inner sep=0.5pt, pos=0.5, fill=white, draw] {\tiny 3};
  \draw[->] (v4) .. controls +(-1.25,-1) and +(-1.25,1) .. (v1) node[circle, inner sep=0.5pt, pos=0.5, fill=white, draw] {\tiny 4};
  \draw[->] (v4) .. controls +(-0.8,-0.75) and +(-0.8, 0.75) .. (v2) node[circle, inner sep=0.5pt, pos=0.6, fill=white, draw] {\tiny 4};
  \draw[->] (v4) .. controls +(-0.175,-0.25) and +(-0.175,0.25) .. (v3) node[circle, inner sep=0.5pt, pos=0.5, fill=white, draw] {\tiny 4};
  \draw[->] (v5) .. controls +(0.4,0.4) and +(0,0.25) .. ++(0.6,0) .. controls +(0,-0.25) and +(0.4,-0.4) .. (v5) node[circle, inner sep=0.5pt, pos=0, fill=white, draw] {\tiny 5};
  \draw[->] (v6) .. controls +(-3,-1.5) and (-3,1.5) .. (v1) node[circle, inner sep=0.5pt, pos=0.5, fill=white, draw] {\tiny 6};
  \draw[->] (v6) .. controls +(-2.25,-1.5) and +(-2.25,1.5) .. (v2) node[circle, inner sep=0.5pt, pos=0.5, fill=white, draw] {\tiny 6};
  \draw[->] (v6) .. controls +(-1.4,-1) and +(-1.4,1) .. (v3) node[circle, inner sep=0.5pt, pos=0.5, fill=white, draw] {\tiny 6};
  \draw[->] (v6)--(v5) node[circle, inner sep=0.5pt, pos=0.5, fill=white, draw] {\tiny 6};
  \draw[->] (v6) .. controls +(0.4,0.4) and +(0,0.25) .. ++(0.6,0) .. controls +(0,-0.25) and +(0.4,-0.4) .. (v6) node[circle, inner sep=0.5pt, pos=0, fill=white, draw] {\tiny 6};
  \draw[->] (v6) .. controls +(0.175,0.25) and +(0.175, -0.25) .. (v7) node[circle, inner sep=0.5pt, pos=0.5, fill=white, draw] {\tiny 6};
  \draw[->] (v7) .. controls +(-0.175,-0.25) and +(-0.175,0.25) .. (v6) node[circle, inner sep=0.5pt, pos=0.5, fill=white, draw] {\tiny 7};
  \draw[->] (v8)--(v7) node[circle, inner sep=0.5pt, pos=0.5, fill=white, draw] {\tiny 8};
  \draw[->] (v8) .. controls +(-0.75,-0.5) and +(-0.75,0.5) .. (v6) node[circle, inner sep=0.5pt, pos=0.5, fill=white, draw] {\tiny 8};
  \draw[->] (v8) .. controls +(-1.25,-0.5) and +(-1.25,0.5) .. (v5) node[circle, inner sep=0.5pt, pos=0.5, fill=white, draw] {\tiny 8};
  \draw[->] (v8) .. controls +(3,-1) and +(3,1) .. (v3) node[circle, inner sep=0.5pt, pos=0.5, fill=white, draw] {\tiny 8};
  \draw[->] (v8) .. controls +(4,-0.5) and +(4,0.5) .. (v2) node[circle, inner sep=0.5pt, pos=0.5, fill=white, draw] {\tiny 8};
  \draw[->] (v8) .. controls +(5,0) and +(5,0) .. (v1) node[circle, inner sep=0.5pt, pos=0.5, fill=white, draw] {\tiny 8};
  \end{scope}
\end{tikzpicture}
\]
\vspace*{-.33in}
\[
\begin{tikzpicture}[xscale=1.5, yscale=1.2]
  \draw[style=dashed] (-0.75,-0.25) arc (171:115:10cm);
  \node[circle, inner sep=-1pt] (1) at (0,0) {\small$\circledcirc$};
  \node[circle, inner sep=6pt] (1nw) at (1.north west) {};
  \node[circle, inner sep=6pt] (1sw) at (1.south west) {};
  \node[circle, inner sep=6pt] (1se) at (1.south east) {};
  \node[anchor=south, inner sep=0pt] (Gamma1) at (1.north) {\tiny 1};
  \begin{scope}[>=latex]
  \node[circle, inner sep=-1pt] (10) at (1,0) {\small$\oplus$};
  \node[anchor=south, inner sep=0pt] at (10.north) {\tiny 10};
  \draw[->>] (1)--(10);
  \node[circle, inner sep=-1pt] (001) at (2,4) {\small$\circledcirc$};
  \node[circle, inner sep=2.5pt] (001se) at (001.south east) {};
  \node[anchor=south, inner sep=0pt] (Gamma2) at (001.north) {\tiny 001};
  \node[circle, inner sep=-1pt] (100) at (2,2) {\small$\oplus$};
  \node[anchor=south, inner sep=0pt] at (100.north) {\tiny 100};
  \draw[->>] (10)--(100);
  \node[circle, inner sep=-1pt] (101) at (2,0) {\small$\oplus$};
  \node[anchor=south, inner sep=0pt] at (101.north) {\tiny 101};
  \draw[->>] (10)--(101);
  \node[circle, inner sep=-1pt] (0010) at (3,4) {\small$\oplus$};
  \node[anchor=south, inner sep=0pt] at (0010.north) {\tiny 0010};
  \draw[->>] (001)--(0010);
  \node[circle, inner sep=-1pt] (1000) at (3,2) {\small$\oplus$};
  \node[anchor=south, inner sep=0pt] at (1000.north) {\tiny 1000};
  \draw[->>] (100)--(1000);
  \node[circle, inner sep=-1pt] (1010) at (3,0) {\small$\oplus$};
  \node[anchor=south, inner sep=0pt] at (1010.north) {\tiny 1010};
  \draw[->>] (101)--(1010);
  \node[circle, inner sep=-1pt] (00001) at (4,6) {\small$\circledcirc$};
  \node[circle, inner sep=2.5pt] (00001se) at (00001.south east) {};
  \node[anchor=south, inner sep=0pt] (Gamma3) at (00001.north) {\tiny 00001};
  \node[circle, inner sep=-1pt] (00100) at (4,5) {\small$\oplus$};
  \node[anchor=south, inner sep=0pt] at (00100.north) {\tiny 00100};
  \draw[->>] (0010)--(00100);
  \node[circle, inner sep=-1pt] (00101) at (4,4) {\small$\oplus$};
  \node[anchor=south, inner sep=0pt] at (00101.north) {\tiny 00101};
  \draw[->>] (0010)--(00101);
  \node[circle, inner sep=-1pt] (10000) at (4,3) {\small$\oplus$};
  \node[anchor=south, inner sep=0pt] at (10000.north) {\tiny 10000};
  \draw[->>] (1000)--(10000);
  \node[circle, inner sep=-1pt] (10001) at (4,2) {\small$\oplus$};
  \node[anchor=south, inner sep=0pt] at (10001.north) {\tiny 10001};
  \draw[->>] (1000)--(10001);
  \node[circle, inner sep=-1pt] (10100) at (4,1) {\small$\oplus$};
  \node[anchor=south, inner sep=0pt] at (10100.north) {\tiny 10100};
  \draw[->>] (1010)--(10100);
  \node[circle, inner sep=-1pt] (10101) at (4,0) {\small$\oplus$};
  \node[anchor=south, inner sep=0pt] at (10101.north) {\tiny 10101};
  \draw[->>] (1010)--(10101);
  \node[circle, inner sep=-1pt] (000010) at (5,6) {\small$\oplus$};
  \node[anchor=south, inner sep=0pt] (000010lab) at (000010.north) {\tiny 000010};
  \draw[->>] (00001)--(000010);
  \node[circle, inner sep=-1pt] (001000) at (5,5) {\small$\oplus$};
  \node[anchor=south, inner sep=0pt] at (001000.north) {\tiny 001000};
  \draw[->>] (00100)--(001000);
  \node[circle, inner sep=-1pt] (001010) at (5,4) {\small$\oplus$};
  \node[anchor=south, inner sep=0pt] at (001010.north) {\tiny 001010};
  \draw[->>] (00101)--(001010);
  \node[circle, inner sep=-1pt] (100000) at (5,3) {\small$\oplus$};
  \node[anchor=south, inner sep=0pt] at (100000.north) {\tiny 100000};
  \draw[->>] (10000)--(100000);
  \node[circle, inner sep=-1pt] (100010) at (5,2) {\small$\oplus$};
  \node[anchor=south, inner sep=0pt] at (100010.north) {\tiny 100010};
  \draw[->>] (10001)--(100010);
  \node[circle, inner sep=-1pt] (101000) at (5,1) {\small$\oplus$};
  \node[anchor=south, inner sep=0pt] at (101000.north) {\tiny 101000};
  \draw[->>] (10100)--(101000);
  \node[circle, inner sep=-1pt] (101010) at (5,0) {\small$\oplus$};
  \node[anchor=south, inner sep=0pt] at (101010.north) {\tiny 101010};
  \draw[->>] (10101)--(101010);
  \node[circle, inner sep=-1pt] (0000100) at (6,6.5) {\small$\oplus$};
  \node[anchor=south, inner sep=0pt] (0000100lab) at (0000100.north) {\tiny 0000100};
  \node[circle, inner sep=1pt] (0000100nwanchor) at (0000100lab.north west) {};
  \node[circle, inner sep=1pt] (0000100neanchor) at (0000100lab.north east) {};
  \draw[->>] (000010)--(0000100);
  \node[circle, inner sep=-1pt] (0010000) at (6,5.5) {\small$\oplus$};
  \node[anchor=south, inner sep=0pt] at (0010000.north) {\tiny 0010000};
  \draw[->>] (001000)--(0010000);
  \node[circle, inner sep=-1pt] (0010100) at (6,4.5) {\small$\oplus$};
  \node[anchor=south, inner sep=0pt] at (0010100.north) {\tiny 0010100};
  \draw[->>] (001010)--(0010100);
  \node[circle, inner sep=-1pt] (1000000) at (6,3.5) {\small$\oplus$};
  \node[anchor=south, inner sep=0pt] at (1000000.north) {\tiny 1000000};
  \draw[->>] (100000)--(1000000);
  \node[circle, inner sep=-1pt] (1000001) at (6,3) {\small$\oplus$};
  \node[anchor=south, inner sep=0pt] at (1000001.north) {\tiny 1000001};
  \draw[->>] (100000)--(1000001);
  \node[circle, inner sep=-1pt] (1000100) at (6,2.5) {\small$\oplus$};
  \node[anchor=south, inner sep=0pt] at (1000100.north) {\tiny 1000100};
  \draw[->>] (100010)--(1000100);
  \node[circle, inner sep=-1pt] (1000101) at (6,2) {\small$\oplus$};
  \node[anchor=south, inner sep=0pt] at (1000101.north) {\tiny 1000101};
  \draw[->>] (100010)--(1000101);
  \node[circle, inner sep=-1pt] (1010001) at (6,1) {\small$\oplus$};
  \node[anchor=south, inner sep=0pt] at (1010001.north) {\tiny 1010001};
  \draw[->>] (101000)--(1010001);
  \node[circle, inner sep=-1pt] (1010101) at (6,0) {\small$\oplus$};
  \node[anchor=south, inner sep=0pt] at (1010101.north) {\tiny 1010101};
  \draw[->>] (101010)--(1010101);
  \node at (6.8,6.5) {\dots};
  \node at (6.8,3.5) {\dots};
  \node at (6.8,0) {\dots};
  \node[circle, inner sep=-1pt] (v1) at (-1.5, 0) {\small$\odot$};
  \node[inner sep=-1pt, anchor=east] at (v1.west) {\tiny $v_1$};
  \node[circle, inner sep=-1pt] (v2) at (-1.5, 1) {\small$\odot$};
  \node[inner sep=-1pt, anchor=east] at (v2.west) {\tiny $v_2$};
  \node[circle, inner sep=-1pt] (v3) at (-1.5,2) {\small$\otimes$};
  \node[inner sep=-1pt, anchor=east] at (v3.west) {\tiny $v_3$};
  \node[circle, inner sep=-1pt] (v4) at (-1.5,3) {\small$\otimes$};
  \node[inner sep=-1pt, anchor=east] at (v4.west) {\tiny $v_4$};
  \node[circle, inner sep=-1pt] (v5) at (-1.5,4) {\small$\otimes$};
  \node[inner sep=-1pt, anchor=east] at (v5.west) {\tiny $v_5$};
  \node[circle, inner sep=-1pt] (v6) at (-1.5,5) {\small$\otimes$};
  \node[inner sep=-1pt, anchor=east] at (v6.west) {\tiny $v_6$};
  \node[circle, inner sep=-1pt] (v7) at (-1.5, 6) {\small$\odot$};
  \node[inner sep=-1pt, anchor=east] at (v7.west) {\tiny $v_7$};
  \node[circle, inner sep=-1pt] (v8) at (-1.5,7) {\small$\otimes$};
  \node[inner sep=-1pt, anchor=east] at (v8.west) {\tiny $v_8$};
  \node at (-1.5,7.5) {$\vdots$};
  \draw[->>] (100)--(v3);
  \draw[->>] (001)--(v4);
  \draw[->>] (00001) .. controls(2.6,5.2) and (0.1,4.4) .. (v5);
  \node[inner sep=2.5pt,circle,anchor=north] (v9bar-control) at (001.south) {};
  \draw[->>] (10000)--(v6);
  \draw[->>] (001)--(v8);
  \end{scope}
  \begin{scope}[>=latex]
  \draw[->] (v1)--(1);
  \draw[->] (v2)--(v1);
  \draw[->] (v3)--(001);
  \draw[->] (v3)--(101);
  \draw[->] (v4) .. controls (-0.9,2) and (-0.9,1) .. (v1);
  \draw[->] (v4) .. controls (-1.1,2) .. (v2);
  \draw[->] (v4)--(v3);
  \draw[->] (v5) .. controls(-0.1,4.8) and (2.4,5.6) .. (00001);
  \draw[->] (v5) .. controls (-0.5,3.6) and (3,3.6) .. (00101);
  \draw[->] (v5)--(10001);
  \draw[->] (v5)--(10101);
  \draw[->] (v6) .. controls (-2.3,4) and (-2.2,1) .. (v1);
  \draw[->] (v6) .. controls (-2.1,4) and (-2.1,2) .. (v2);
  \draw[->] (v6) .. controls (-1.9,4) and (-1.9,3) .. (v3);
  \draw[->] (v6)--(v5);
  \draw[->] (v6) .. controls (-1,5.6) and (-1,4.4)  .. (v6);
  \draw[->] (v6) .. controls (-1.4,5.4) and (-1.4,5.6) .. (v7);
  \draw[->] (v7) .. controls (-1.6,5.6) and (-1.6,5.4) .. (v6);
  \draw[->] (v7)--(1000001);
  \draw[->] (v7) .. controls (4.5,3) .. (1000101);
  \draw[->] (v7)--(1010001);
  \draw[->] (v7)--(1010101);
  \draw[->] (v8)--(v7);
  \draw[->] (v8) .. controls (-1.7,6.4) and (-1.8,5.5) .. (v6);
  \draw[->] (v8) .. controls (-1.9,6) and (-1.9,5) .. (v5);
  \draw[->] (v8) .. controls (-2.1,6) and (-2.1,3) .. (v3);
  \draw[->] (v8) .. controls (-2.25,6) and (-2.25,2) .. (v2);
  \draw[->] (v8) .. controls (-2.4,6) and (-2.4,1) .. (v1);
  \end{scope}
\end{tikzpicture}
\]
\caption{The ultragraph $\G$ (top) and graph $E$ (bottom) of
Example~\ref{ex:illustration}}\label{pic:G,E}
\end{figure}

We now describe $W_+$. By Lemma~\ref{lem:W_+=disjoint union},
$W_+$ is the disjoint union of the sets $r(0^{2(p-3)}1)$ where
$p$ runs through $4$ and all odd prime numbers. We have
\begin{align*}
r(1)&=\{v_m : 3\mid m\}, \\
r(001)&= \{v_m : 3 \nmid m\text{ and } 4 \mid m\}, \\
r(00001)&= \{v_m : 3 \nmid m,\ 4 \nmid m, \text{ and } 5 \mid m\},
\intertext{and for an odd prime number $p$ greater than $5$,}
r(0^{2(p-3)}1) &= \big\{v_m :
3 \nmid m,\ 4 \nmid m, \ldots, (p-1) \nmid m,\ p \mid m, \text{ and }
m > (p-3)^2\big\}.
\end{align*}
This implies that $v_1, v_2 \not\in W_+$ and that $v_m \in W_+$
whenever $3\mid m$, $4\mid m$, or $5\mid m$. Fix $m\in
\N\setminus (\{1,2\}\cup 3\N \cup 4\N \cup 5\N)$. Let $p$ be
the smallest odd prime divisor of $m$. Then $p$ is
greater than $5$. Moreover $v_m \in W_+$ if and only if $v_m
\in r(0^{2(p-3)}1)$, which is equivalent to $m > (p-3)^2$. Let
$k = m/p \in \N$. Since $p$ is the smallest odd prime divisor of $m$, either $k=1$, $k=2$, or $k\geq p$. If
$k=1$ or $k=2$, we have $m = kp \leq (p-3)^2$ and hence $v_m
\not\in W_+$. If $k\geq p$, then $m = kp \geq p^2 > (p-3)^2$,
so $v_m \in W_+$.
Recall that $W^0 = G^0 \setminus W_+$.
We have proved that $W_0$ may be described as
\begin{align*}
W_0
&= \big\{v_p,v_{2p} :
\text{$p=1$ or $p$ is an odd prime number greater than $5$}\big\}\\
&= \{v_1,v_2, v_7, v_{11}, v_{13}, v_{14},
v_{17}, v_{19}, v_{22}, v_{23}, \ldots \},
\end{align*}
and then $W_+$ is the complement of this set:
\[
W_+ = G^0 \setminus W_0 = \{v_3, v_4, v_5, v_6, v_8, v_9, v_{10}, v_{12}, v_{15}, v_{16}, v_{18}, v_{20}, v_{21}, \ldots\}.
\]

We now define a function $\sigma\colon W_+\to \Delta$ with the
properties described in Lemma~\ref{lem:W_+ -> Delta^0}.  Since each $r(e_{2k}) = \{v_n : n \le k^2, 4\nmid n\}$, the set $W_\infty \subset W_+$ described in the proof of
Lemma~\ref{lem:W_+ -> Delta^0} is $\{v_m \in W_+: 4 \mid m\}$.
Thus $\{v_{4}, v_8, v_{12}, v_{16}, \dots\}$ is an ordering of
$W_\infty$. For $k \in \N$, let $n := \max\{3,k\}$. Then $n$ is
the smallest integer such that $n \ge k$ and $v_{4k} \in
\bigsqcup_{\omega \in \Delta_n} r(\omega)$. Define
$\sigma(v_{4k})$ to be the unique element $\omega$ of
$\Delta_n$ such that $v_{4k} \in r(\omega)$. So
\[
\sigma(v_4) = \sigma(v_8) = 001, \quad
\sigma(v_{12}) = 101, \quad
\sigma(v_{16}) = 0010, \quad
\sigma(v_{20}) = 00101, \quad \dots
\]
For $v_m \in W_+ \setminus W_{\infty}$, let $k$ be the minimal
integer such that $m \le k^2$. Then $n := 2k - 1$ is the
maximal integer such that $v_m \in \bigsqcup_{\omega \in
\Delta_n} r(\omega)$. We define $\sigma(v_m)$ to be the unique
element $\omega$ of $\Delta_n$ such that $v_m \in r(\omega)$.
So
\begin{align*}
\sigma(v_3) &= 100, \quad \sigma(v_5) = 00001, \quad
\sigma(v_6) = 10000,\quad \sigma(v_9) = 10000,\\
\sigma(v_{10}) &= 0000100,\quad
\sigma(v_{15}) = 1000100,\quad \sigma(v_{18}) = 100000100,\quad \dots
\end{align*}
By our convention that $\sigma(v) = \emptyset$
whenever $v \in W_0$,
we have
\[
\emptyset = \sigma(v_1) = \sigma(v_2)
= \sigma(v_7) = \sigma(v_{11}) = \sigma(v_{13})
= \sigma(v_{14}) = \sigma(v_{17}) = \sigma(v_{19}) = \cdots.
\]
We also have
\begin{gather*}
\Xset{1} = \{1\},\qquad \Xset{2} = \{v_1\},\qquad
\Xset{3} = \{001,101\},\qquad \Xset{4} = \{v_1, v_2, v_3\}, \\
\Xset{5} = \{00001,00101,10001,10101\},\qquad
\Xset{6} = \{v_1, v_2, v_3, v_5, v_6, v_7, v_9\}, \\
\Xset{7} = \{v_6, v_{12}, v_{24}, 1000001, 1000101, 1010001, 1010101\}, \\
\Xset{8} = \{v_1, v_2, v_3, v_5, v_6, v_7, v_9,
      v_{10}, v_{11}, v_{13}, v_{14}, v_{15}\},\qquad \dots
\end{gather*}
This is all the information required to draw $E$, and we have
done so in Figure~\ref{pic:G,E}. To distinguish the various
special sets of vertices discussed above, we draw vertices
using four different symbols as follows: vertices of the form
$\odot$ belong to $W_0$; those of the form $\otimes$ belong to
$W_+$; those of the form $\circledcirc$ belong to $\Gamma_0$;
and those of the form $\oplus$ belong to $\Gamma_+$. The dashed
arc separates $G^0$ on the left from $\Delta$ on the right.

Edges drawn as double-headed arrows are of the form
$\overline{x}$ where $x \in W_+\sqcup\Gamma_+$, and edges drawn
as single-headed arrows are of the form $\edge(n,x)$ where $e_n
\in \G^1$ and $x \in \Xset{n}$. Since $s_E(\overline{x}) \in
\Delta$ and $r_E(\overline{x}) = x$ for all $x \in
W_+\sqcup\Gamma_+$, and since $s_E(\edge(n,x)) = s(e_n) = v_n$
and $r_E(\edge(n,x)) = x$ for all $n$ and $x \in \Xset{n}$,
once we know the type of an edge, the edge is uniquely
determined by its source and its range. Thus it is not
necessary to label the edges in the figure.
\end{example}

\section{Paths and Condition~(K)} \label{Path-sec}
\setcounter{footnote}{1}

For this section, we fix an ultragraph $\G$, and make a choice
of an ordering $\{e_1, e_2, \dots\}$ of $\G^1$ and a function
$\sigma$ as in Lemma~\ref{lem:W_+ -> Delta^0}. Let
$E=(E^0,E^1,r_E,s_E)$ be the graph constructed from $\G$ as in
Definition~\ref{dfn:E}. We relate the path structure of $E$ to
that of $\G$. In particular, we show that $\G$ satisfies
Condition~(K) as in \cite{KMST} if and only if $E$ satisfies
Condition~(K) as in \cite{KPRR}. Condition~(K) was introduced
in \cite{KPRR} to characterise those graphs in whose
$C^*$-algebras every ideal is gauge-invariant. In
Section~\ref{ideal-sec}, we will combine our results in this
section with our main result Theorem~\ref{thm:fullcorner} to
deduce from Kumjian, Pask, Raeburn and Renault's result the
corresponding theorem for ultragraph $C^*$-algebras.

We recall some terminology for graphs (see, for example,
\cite{KPR, BPRS2000}; note that our edge direction convention
agrees with that used in these papers and in \cite{KMST}, which
is opposite to that used in \cite{Raeburn2005}).

For each integer $n \geq 2$, we write
\[
E^n := \{\alpha = \alpha_1 \alpha_2 \cdots \alpha_n :
\text{$\alpha_i \in E^1$ and $s_E(\alpha_{i+1})=r_E(\alpha_{i})$ for all $i$}\},
\]
and $E^* := \bigsqcup_{n=0}^\infty E^n$. The elements of $E^*$
are called \emph{paths}. The \emph{length} of a path $\alpha$
is the integer $|\alpha|$ such that $\alpha \in E^{|\alpha|}$.
We extend the range and source maps to $E^*$ as follows. For $v
\in E^0$, we write $r_E(v) = s_E(v) = v$. For $\alpha \in
E^*\setminus E^0$ we write $r_E(\alpha) =
r_E(\alpha_{|\alpha|})$ and $s_E(\alpha) = s_E(\alpha_1)$.

For $\alpha,\beta \in E^*$ such that $r_E(\alpha) =
s_E(\beta)$, we may form the path $\alpha\beta \in E^*$ by
concatenation. Thus $E^*$ becomes a category whose unit space
is $E^0\subset E^*$. For a vertex $v \in E^0$, a \emph{return
path based at $v$} is a path $\alpha$ of nonzero length with
$s_E(\alpha) = r_E(\alpha) = v$. A return path $\alpha =
\alpha_1 \alpha_2 \cdots \alpha_n$ based at $v$ is called a
\emph{first-return path} if $s_E(\alpha_i) \neq v$ for $i = 2,
3, \ldots, n$. We say that $E$ satisfies \emph{Condition~(K)}
if no vertex is the base of exactly one first-return path
(equivalently, each vertex is either the base of no
return path, or is the base of at least two first-return
paths).

Consider the subgraph $F$ of $E$ with the same vertices as $E$,
and edges $F^1=\{\overline{x} : x\in W_+\sqcup\Gamma_+\}
\subset E^1$. Note that $F^1\subset E^1$ is the set of all
edges in $E^1$ starting from $\Delta \subset E^0$. In Figure~\ref{pic:G,E} the elements of
$F^1$ are the double-headed arrows. Then
\[
F^*=E^0\sqcup\big\{\overline{x}_1\overline{x}_2\cdots \overline{x}_n\in E^*:
n\in\N, x_i\in W_+\sqcup\Gamma_+\big\}\subset E^*.
\]
In particular $\alpha \in E^*$ belongs to $F^*$ if and only if
it contains no edges of the form $\edge(n,x)$ where
$e_n\in\G^1$ and $x\in \Xset{n}$.

\begin{lemma}\label{lem:unique expression}
Every $\alpha\in E^*$ can be uniquely expressed as
\[
\alpha=g_0\cdot \edge(n_1,x_1)\cdot g_1\cdot \edge(n_2,x_2)\cdot g_{2}
\cdot \cdots \cdot \edge(n_k,x_k)\cdot g_{k}
\]
where each $e_{n_i}\in\G^1$, $x_i\in \Xset{n_i}$ and $g_i\in F^*$.
\end{lemma}

\begin{proof}
Let $\alpha = \alpha_1\alpha_2 \dots \alpha_n$. Whenever
$\alpha_i$ and $\alpha_{i+1}$ are both of the form
$\edge(n,x)$, rewrite $\alpha_i \alpha_{i+1} = \alpha_i
r_E(\alpha_i) \alpha_{i+1}$ (recall that $r_E(\alpha_i) \in
E^0$ belongs to $F^*$ by definition). Now by grouping sequences
of consecutive edges from $F^1$, we obtain an expression for
$\alpha$ of the desired form. This expression is clearly
unique.
\end{proof}

In the graph $F$, we can distinguish the sets $W_0$, $W_+$,
$\Gamma_0$, and $\Gamma_+$ as follows.
\begin{itemize}
\item An element in $W_0$ emits no edges,
and receives no edges.
\item An element in $W_+$ emits no edges,
and receives exactly one edge.
\item An element in $\Gamma_0$ emits a finite and nonzero
  number of edges, and receives no edges.
\item An element in $\Gamma_+$ emits a finite and nonzero
  number of edges, and receives exactly one edges.
\end{itemize}
Next we describe the paths of the graph $F$. To do so, the
following notation is useful.

\begin{definition}
For $n \in \N$ and $\omega \in \{0,1\}^n \setminus \{0^n\}$ we set
\[
r'(\omega)
:= \{v \in r(\omega) : |\sigma(v)| \geq n\}
=\{ v \in G^0 : |\sigma(v)| \geq n,\ \sigma(v)|_{n} = \omega\}.
\]
\end{definition}

To see that the two sets in the definition above coincide, it
suffices to see that for $v \in G^0$ with $|\sigma(v)| \ge n$,
we have $v \in r(\omega)$ if and only if $\sigma(v)|_{n} =
\omega$. For this, observe that $v \in r(\sigma(v)) \subset
r(\sigma(v)|_{n})$ and that there is at most one $\omega \in
\{0,1\}^{n} \setminus \{0^{n}\}$ such that $v \in r(\omega)$.

\begin{lemma}\label{lem:r'empty}
For $\omega \notin \Delta$ the set $r'(\omega)$ is empty.
\end{lemma}

\begin{proof}
If $\omega \in \{0,1\}^n \setminus \{0^n\}$
for some $n \in \N$, and suppose that $r'(\omega)$ is not empty.
Then there exists $v\in G^0$ with $|\sigma(v)|\geq n$
and $\sigma(v)|_{n} = \omega$.
Since $\sigma(v)\in\Delta$ and $\omega\neq 0^n$,
we have $\omega \in \Delta$.
This shows $r'(\omega) =\emptyset$ for $\omega \notin \Delta$.
\end{proof}

\begin{lemma}\label{lem:r(e_n) and X_n}
For each $n \in \N$, we have $r(e_n) = (\Xset{n} \cap G^0)
\sqcup \big( \bigsqcup_{\omega \in \Xset{n} \cap \Delta}r'(\omega) \big)$.
\end{lemma}
\begin{proof}
We have $r(e_n) =  \bigsqcup_{\omega \in
\{0,1\}^n,\,\omega_n=1} r(\omega)$. The definition of
$\Xset{n}$ (see Definition~\ref{def:X_n}) guarantees that
$\Xset{n} \cap G^0 = \{v \in r(e_n) : |\sigma(v)| < n\}$. For
$\omega \in \{0,1\}^n$ with $\omega_n=1$, we have
$r'(\omega)=r(\omega)\setminus\{v \in r(\omega) : |\sigma(v)| <
n\}$ by definition. Hence
\[\textstyle
r(e_n) = (\Xset{n} \cap G^0) \sqcup \big(\bigsqcup_{\omega \in \{0,1\}^n,\,\omega_n=1}
r'(\omega)\big).
\]
Finally $r'(\omega)=\emptyset$ for $\omega \notin \Delta$ by
Lemma~\ref{lem:r'empty}.
\end{proof}

\begin{remark}\label{rem:r'=r'}
For $\omega \in \{0,1\}^n \setminus \{0^n\}$
one can show
$r'(\omega) = r'(\omega, 0) \sqcup r'(\omega, 1) \sqcup \sigma^{-1}(\omega)$,
using the fact $\sigma^{-1}(\omega)=\{v \in r(\omega) : |\sigma(v)|=n\}$.
We omit the routine proof because we do not use it,
but we remark this fact because this relates to Lemma~\ref{lem:Q'}
(this can be proved using Lemma~\ref{lem:paths in F}~(3) below).
\end{remark}

\begin{lemma}\label{lem:paths in F}
The graph $F$ contains no return paths,
and each $\alpha \in F^*$ is uniquely determined by $s_E(\alpha)$
and $r_E(\alpha)$. Moreover,
\begin{enumerate}
\item every path in $\alpha \in F$ of nonzero length satisfies $s_E(\alpha) \in \Delta$;
\item there is a path in $F$ from $\omega \in \Delta$ to
  $\omega' \in \Delta$ if and only if $|\omega| \le
  |\omega'|$ and $\omega = \omega'|_{|\omega|}$.
\item there is a path in $F$ from $\omega \in \Delta$ to $v
  \in G^0$ if and only if $v \in r'(\omega)$; and
\end{enumerate}
\end{lemma}
\begin{proof}
Fix $e \in F^1$.
Then either $r_E(e) \in G^0$ and hence is a sink in $F$,
or else $s_E(e) \in \Delta_n$ and $r_E(e) \in
\Delta_{n+1}$ for some $n \in \N$.
Thus $F$ contains no return paths.

Now suppose that $\alpha, \alpha' \in F^*$ satisfy $r_E(\alpha) =
r_E(\alpha')$ and $s_E(\alpha) = s_E(\alpha')$.
Without loss of generality, we may assume that $|\alpha| \ge |\alpha'|$.
By definition of $F$, each vertex $v \in E^0$ receives at most one
edge in $F^1$, so $\alpha = \beta\alpha'$ for some $\beta \in F^*$.
This forces $s_E(\beta) = s_E(\alpha) = s_E(\alpha') = r_E(\beta)$,
and then $\beta$ has length $0$ by the preceding paragraph,
and $\alpha = \alpha'$.

By definition of $F$, we have $s_E(F^1) = \Delta$, which
proves~(1). As explained in the first paragraph, a path
$\alpha$ from $\Delta_{n}$ to $\omega'\in \Delta$ must have the
form $\alpha = \overline{\omega'|_{n+1}} \cdot
\overline{\omega'|_{n+2}} \cdots
\overline{\omega'|_{|\omega'|-1}} \cdot \overline{\omega'}$.
This expression makes sense if and only if $n \le |\omega'|$
and $\omega := \omega'|_n$ is in $\Delta_{n}$, and then
$\alpha$ has source $\omega$. This proves~(2). For~(3), fix
$\omega \in \Delta_{n}$ and $v \in G^0$. There is a path from
$\omega$ to $v$ if and only if $v \in W_+$ and there is a path
from $\omega$ to $\sigma(v)$. By~(2), this occurs if and only
if $n \le |\sigma(v)|$ and $\sigma(v)|_{n} = \omega$ (in
particular, $n=|\omega|$). Thus, there is a path from $\omega$
to $v$ if and only if $v \in r'(\omega)$.
\end{proof}

\begin{definition}\label{dfn:f-paths}
Lemma~\ref{lem:paths in F} implies that for each $x \in E^0$,
there is a unique element $f_x \in F^*$ such that $r_E(f_x) = x$ and
$s_E(f_x) \in W_0 \sqcup \Gamma_0$. Observe that
\begin{itemize}
\item For $x\in W_0\sqcup\Gamma_0$, we have $f_x = x$.
\item For $x = \omega\in\Gamma_+$, we have
\[
f_x:=\overline{\omega|_{m+1}}\cdot\overline{\omega|_{m+2}}\cdots
\overline{\omega|_{|\omega|-1}}\cdot\overline{\omega}
\]
where $m=\min\{k:\omega_k=1\}$.
\item For $x = v\in W_+$, we have $f_x = f_{\sigma(v)}
  \overline{v}$.
\end{itemize}
\end{definition}

\begin{example}
Consider the ultragraph of Example~\ref{ex:illustration}, and
the corresponding graph $E$ illustrated there.
\begin{itemize}
\item We have $f_{v_1} = v_1$ and $f_{00001} = 00001$
since $v_1 \in W_0$ and $00001 \in \Gamma_0$.
\item We have $f_{001000} =
  \overline{0010}\cdot\overline{00100}\cdot\overline{001000}$.
\item We have $f_{v_6} =
  \overline{10}\cdot\overline{100}\cdot\overline{1000}\cdot\overline{10000}\cdot\overline{v_6}$.
\end{itemize}
In the second two instances, it is easy to see that $f_x$ is
the unique path in double-headed arrows from $\Gamma_0$ (that
is, a vertex of the form $\circledcirc$) to $x$.
\end{example}

\begin{lemma}\label{lem:path}
For fixed $v,w \in G^0$, the map
\begin{equation}
g_0 \cdot \edge(n_1,x_1)\cdot g_1\cdot
\edge(n_2,x_2)\cdot g_{2}\cdots \edge(n_k,x_k)\cdot g_{k}
\mapsto \begin{cases}
g_0 &\text{ if $k = 0$} \\
e_{n_1}e_{n_2}\cdots e_{n_k} &\text{ otherwise}
\end{cases}
\end{equation}
is a bijection between paths in $E$ from $v$ to $w$ and
paths in $\G$ beginning at $v$ whose ranges contain $w$
where each $e_{n_i}\in\G^1$, $x_i\in \Xset{n_i}$ and $g_i\in F^*$
as in Lemma~\ref{lem:unique expression}.
\end{lemma}

\begin{proof}
First note that we have $g_0=v$ because $v\in G^0$
emits no edges in $F$.
Since $s_E(\edge(n_i,x_i)) = s(e_{n_i}) \in G^0$,
to show that the map is well-defined and bijection,
it suffices to show that
for each $e_{n}\in\G^1$ and $w \in G^0$
there exists a path $\alpha = \edge(n,x)\cdot g$
where $x\in \Xset{n}$ and $g\in F^*$ satisfying $r_E(\edge(n,x)) = w$
if and only if $w \in r(e_n)$,
and in this case $x\in \Xset{n}$ and $g\in F^*$ are unique.
This follows from
Lemma~\ref{lem:r(e_n) and X_n} and Lemma~\ref{lem:paths in F}~(3).
\end{proof}

We introduced notions of paths and Condition~(K) for graphs at
the beginning of this section. We now recall the corresponding
notions for ultragraphs. A \emph{path} in an ultragraph is a
sequence $\alpha = \alpha_1 \alpha_2 \dots \alpha_{|\alpha|}$
of edges such that $s(\alpha_{i+1}) \in r(\alpha_i)$ for all
$i$. We write $s(\alpha) = s(\alpha_1)$ and $r(\alpha) =
r(\alpha_{|\alpha|})$. A \emph{return path} is a path $\alpha$
such that $s(\alpha) \in r(\alpha)$. A \emph{first-return path}
is a return path $\alpha$ such that $s(\alpha) \not=
s(\alpha_i)$ for any $i \ge 1$. As in \cite[Section~7]{KMST},
we say an ultragraph $\G = (G^0, \G^1, r, s)$ satisfies
\emph{Condition~(K)} if no vertex is the base of exactly one
first-return path.

\begin{proposition}\label{prop:Cond(K)}
The graph $E$ satisfies Condition~(K) if and only if the
ultragraph $\G$ satisfies Condition~(K).
\end{proposition}
\begin{proof}
Lemma~\ref{lem:paths in F} implies that every return path in
$E$ passes through some vertex in $G^0$. Hence $E$ satisfies
Condition~(K) if and only if no vertex in $G^0$ is the base of
exactly one first-return path in $E$. This in turn happens if
and only if $\G$ satisfies Condition~(K) by
Lemma~\ref{lem:path}.
\end{proof}

\section{Full corners of graph algebras} \label{corner-sec}

Once again, we fix an ultragraph $\G$ and a graph $E$
constructed from $\G$ as in Definition~\ref{dfn:E}. We will
show that the ultragraph algebra $C^*(\G)$ is isomorphic to a
full corner of the graph algebra $C^*(E)$.

\begin{definition}\label{def:graph algebra}
The \emph{graph algebra} $C^*(E)$ of the graph
$E=(E^0,E^1,r_E,s_E)$ is the universal \Ca generated by
mutually orthogonal projections $\{q_x : x\in E^0\}$ and
partial isometries $\{t_\alpha : \alpha\in E^1\}$ with mutually
orthogonal ranges satisfying the Cuntz-Krieger relations:
\begin{enumerate}
\item $t_\alpha^*t_\alpha = q_{r_E(\alpha)}$ for all
  $\alpha\in E^1$;
\item $t_\alpha t_\alpha^* \leq q_{s_E(\alpha)}$ for all
  $\alpha\in E^1$; and
\item $q_{x} = \sum_{s_E(\alpha)=x} t_\alpha t_\alpha^*$
for $x \in E^0_{\rg}$.
\end{enumerate}
\end{definition}

As usual, for a path $\alpha=\alpha_1\alpha_2\cdots \alpha_n$ in $E$
we define $t_\alpha\in C^*(E)$
by $t_\alpha=t_{\alpha_1}t_{\alpha_2}\cdots t_{\alpha_n}$.
For $x \in E^0 \subset E^*$,
the notation $t_x$ is understood as $q_x$.
The properties (1) and (2) in Definition~\ref{def:graph algebra}
hold for all $\alpha \in E^*$.

\begin{definition}\label{dfn:Ux}
For each $x\in E^0$ define a partial isometry $U_x :=
t_{f_x}\in C^*(E)$ where $f_x \in F^*$ is as in
Definition~\ref{dfn:f-paths}.
\end{definition}

By definition of $f_x$ and the Cuntz-Krieger relations, we have
$U_x^*U_x=q_x$ and $U_xU_x^*\leq q_{s_E(f_x)}$ for $x\in E^0$.

\begin{lemma}\label{lem:uxux*uyuy*}
For $x,y\in E^0$ with $x\neq y$, we have
\[
(U_{x}U_{x}^*)(U_{y}U_{y}^*)
=\begin{cases}
U_{x}U_{x}^* & \text{if there exists a path in $F$ from $x$ to $y$,}\\
U_{y}U_{y}^* & \text{if there exists a path in $F$ from $y$ to $x$,}\\
0 & \text{otherwise.}
\end{cases}
\]
In particular, for $n\in\N$ and $x,y\in \Xset{n}$ with $x\neq
y$, we have $U_x^*U_y=0$.
\end{lemma}
\begin{proof}
Without loss of generality, we may assume $|f_x| \le |f_y|$.
Then $U_{x}^*U_{y}\neq 0$ if and only if $f_y$ extends $f_x$,
and in this case $(U_{x}U_{x}^*)U_{y}=U_{y}$. By the uniqueness
of $f_y$ in $F^*$ stated in Definition~\ref{dfn:f-paths}, $f_y$
extends $f_x$ exactly when there exists a path in $F$ from $x$
to $y$.

For the last statement, observe that by Lemma~\ref{lem:r(e_n)
and X_n} and Lemma~\ref{lem:paths in F}, there exist no paths
in $F$ among vertices in $\Xset{n}$. Hence $U_x^*U_y =
U_x^*U_xU^*_x U_y U^*_y U_y = 0$.
\end{proof}

\begin{definition}\label{dfn:G-family in E}
For $v\in G^0$, we set $P_v := U_vU_v^*$. For $e_n\in\G^1$, we
set
\[
S_{e_n} := U_{s(e_n)}\sum_{x\in \Xset{n}}t_{\edge(n,x)}U_{x}^*.
\]
\end{definition}

It is clear that $P_v$ is a nonzero projection, and the last
statement of Lemma~\ref{lem:uxux*uyuy*} implies that $S_{e_n}$
is a partial isometry. We will show in Proposition
\ref{prop:ELGfam} that the collection $\{P_v : v\in G^0\}$ and
$\{S_{e_n} : e_n\in\G^1\}$ is an Exel-Laca $\G$-family in
$C^*(E)$.

\begin{definition}\label{dfn:Q'}
For $e_n \in \G^1$, we define $Q_{e_n}:=S_{e_n}^*S_{e_n}\in
C^*(E)$. For $\omega \in \bigsqcup_{n=1}^\infty (\{0,1\}^n
\setminus \{0^n\})$, we define
\[
Q'_\omega := \begin{cases}
U_{\omega}U_{\omega}^* &\text{ if $\omega \in \Delta$} \\
0&\text{ otherwise.}
\end{cases}
\]
\end{definition}

The projections $Q'_\omega$ are related to the sets
$r'(\omega)$ of the preceding section (see
Proposition~\ref{prop:r'(o)}).

\begin{lemma}\label{lem:P_vQ'_o}
The collections $\{P_v : v\in G^0\}$ and $\{Q_\omega' :
\omega\in \bigsqcup_{n=1}^\infty (\{0,1\}^n \setminus
\{0^n\})\}$ of projections satisfy the following:
\begin{enumerate}
\item $\{P_v : v\in G^0\}$ are pairwise orthogonal.
\item $\{Q_\omega' : \omega\in \bigsqcup_{n=1}^\infty
    (\{0,1\}^n \setminus \{0^n\})\}$ pairwise commute.
\item $\{Q_\omega' : \omega\in \{0,1\}^n \setminus
    \{0^n\}\}$ are pairwise orthogonal for each $n \in \N$.
\item $P_v Q_\omega' = Q_\omega' P_v = P_v$ if $v \in r'(\omega)$,
and $P_v Q_\omega' = Q_\omega' P_v = 0$ if $v \notin r'(\omega)$.
\end{enumerate}
\end{lemma}
\begin{proof}
By Lemma~\ref{lem:paths in F} paths in $F$ are uniquely
determined by their ranges and sources, and
Lemma~\ref{lem:uxux*uyuy*} shows how the $U_x U^*_x$ multiply.
The four assertions follow immediately.
\end{proof}

\begin{lemma}\label{lem:Q_e=}
For $n\in\N$, we have
\[
Q_{e_n}=\sum_{x\in \Xset{n}}U_xU_x^*=
\sum_{\substack{\omega\in \{0,1\}^{n}\\ \omega_n=1}}Q'_\omega
+\sum_{\substack{v\in r(e_{n})\\ |\sigma(v)|<n}}P_v.
\]
\end{lemma}
\begin{proof}
We compute:
\begin{align*}
Q_{e_n}
&= S_{e_n}^* S_{e_n} \\
&=\bigg(\sum_{x\in \Xset{n}}U_xt_{\edge(n,x)}^*\bigg)U_{s(e_n)}^*U_{s(e_n)}
\bigg(\sum_{y\in \Xset{n}}t_{\edge(n,y)}U_y^*\bigg) \\
&=\sum_{x,y\in \Xset{n}}(U_xt_{\edge(n,x)}^*t_{\edge(n,y)}U_y^*).
\end{align*}
Since $t_{\edge(n,x)}^*t_{\edge(n,y)}=0$ for $x,y\in \Xset{n}$
with $x\neq y$, we deduce that $Q_{e_n}=\sum_{x\in
\Xset{n}}U_xU_x^*$ as claimed.

By the definition of $\Xset{n}$, we have
\[
\sum_{x\in \Xset{n}}U_xU_x^*=
\sum_{\substack{\omega\in \{0,1\}^{n}\\ \omega_n=1}}Q'_\omega
+\sum_{\substack{v\in r(e_{n})\\ |\sigma(v)|<n}}P_v.\qedhere
\]
\end{proof}

\begin{lemma}\label{lem:Q_e}
The collection of projections $\{Q_{e} : e\in\G^1\}$ satisfy
the following:
\begin{enumerate}
\item $\{Q_{e} : e\in\G^1\}$ pairwise commute.
\item $P_vQ_e=Q_eP_v=P_v$ if $v\in r(e)$,
and $P_vQ_e=Q_eP_v=0$ if $v\notin r(e)$.
\item For $n \in \N$ and $\omega \in \{0,1\}^n \setminus
  \{0^n\}$, we have $Q_\omega'Q_{e_n} = Q_{e_n}Q_\omega'
  = Q_\omega'$ if $\omega_n=1$, and $Q_\omega'Q_{e_n} =
  Q_{e_n}Q_\omega' = 0$ if $\omega_n=0$.
\end{enumerate}
\end{lemma}
\begin{proof}
Assertions (1)~and~(2) follow from routine calculations using
Lemma~\ref{lem:P_vQ'_o} and Lemma~\ref{lem:Q_e=}. Assertion~(3)
follows from similar calculations using the decomposition of
$r(e_n)$ from Lemma~\ref{lem:r(e_n) and X_n}.
\end{proof}

\begin{lemma}\label{lem:Q'}
For $\omega\in \{0,1\}^n\setminus\{0^n\}$,
we have
\[
Q'_\omega=Q'_{(\omega, 0)}+Q'_{(\omega, 1)}
+\sum_{\substack{v\in r(\omega)\\ |\sigma(v)|=n}}P_v.
\]
\end{lemma}
\begin{proof}
For $\omega\notin\Delta$ both sides of the equation are zero.
For $\omega\in\Delta$, we have $\omega\in E^0_{\rg}$ by
Proposition~\ref{prp:reg verts}. Hence by the Cuntz-Krieger
relations, we have
\begin{align*}
q_\omega
&=\sum_{\substack{i\in \{0,1\}\\ (\omega, i)\in\Delta}}
t_{\overline{(\omega, i)}}t_{\overline{(\omega, i)}}^*
+\sum_{\substack{v\in G^0\\ \sigma(v)=\omega}}
t_{\overline{v}}t_{\overline{v}}^*\\
&=\sum_{\substack{i\in \{0,1\}\\ (\omega, i)\in\Delta}}
t_{\overline{(\omega, i)}}t_{\overline{(\omega, i)}}^*
+\sum_{\substack{v\in r(\omega)\\ |\sigma(v)|=n}}
t_{\overline{v}}t_{\overline{v}}^*.
\end{align*}
Multiplying by $U_\omega$ on the left and by $U_\omega^*$ on
the right gives the desired equation.
\end{proof}

\begin{definition}\label{dfn:Q_o}
For $n\in\N$ and $\omega\in\{0,1\}^n\setminus\{0^n\}$, we
define $Q_\omega\in C^*(E)$ by
\[
Q_\omega := \prod_{\omega_i=1}Q_{e_i}\prod_{\omega_j=0}(1-Q_{e_j}).
\]
\end{definition}

\begin{lemma}\label{lem:Q_o}
For every $\omega \in \{0,1\}^n\setminus\{0^n\}$, we have
\begin{equation}\label{eq:Q(1-Q)}
Q_\omega = Q'_\omega +
\sum_{\substack{v\in r(\omega)\\ |\sigma(v)|<|\omega|}}P_v.
\end{equation}
\end{lemma}
\begin{proof}
We proceed by induction on $n$. The case $n=1$ follows from
Lemma \ref{lem:Q_e=} because $Q_{\omega}=Q_{e_1}$ and
$r(\omega) = r(e_1)$ for the only element $\omega=(1)$ of
$\{0,1\}^1\setminus\{0^1\}$.

Fix $n \in \N$, and suppose as an inductive hypothesis that
Equation~\eqref{eq:Q(1-Q)} holds for all elements of
$\{0,1\}^{n}\setminus\{0^{n}\}$. Then for each $\theta \in
\{0,1\}^{n}\setminus\{0^{n}\}$, the inductive hypothesis and
Lemma~\ref{lem:Q'} imply that
\begin{align}
Q_\theta
&= Q'_\theta + \sum_{\substack{v\in r(\theta)\\ |\sigma(v)|<n}}P_v \nonumber \\
&= \bigg(Q'_{(\theta, 0)}+Q'_{(\theta, 1)}
+\sum_{\substack{v\in r(\theta)\\ |\sigma(v)|=n}}P_v \bigg)
+\sum_{\substack{v\in r(\theta)\\ |\sigma(v)|<n}}P_v \nonumber \\
&= Q'_{(\theta, 0)} + Q'_{(\theta, 1)}
+ \sum_{\substack{v\in r(\theta)\\ |\sigma(v)|<n+1}}P_v.\label{eq:Q from Q'}
\end{align}

Now fix $\omega \in \{0,1\}^{n+1} \setminus \{0^{n+1}\}$; we
must establish Equation~\ref{eq:Q(1-Q)} for this $\omega$. We
consider three cases: $\omega = (\theta,1)$ for some $\theta
\in \{0,1\}^n \setminus \{0^n\}$; $\omega = (\theta,0)$ for
some $\theta \in \{0,1\}^n \setminus \{0^n\}$; or $\omega =
(0^n,1)$.

First suppose that $\omega = (\theta,1)$. Then $Q_\omega =
Q_{(\theta, 1)} =Q_\theta Q_{e_{n+1}}$. Combining this with
Lemma~\ref{lem:Q_e} (2)~and~(3) and with~\ref{eq:Q from Q'}, we
obtain
\[
Q_\omega
=Q'_{(\theta, 1)} +
\sum_{\substack{v\in r(\theta)\cap r(e_{n+1})\\ |\sigma(v)|<n+1}}P_v
=Q'_{(\theta, 1)} + \sum_{\substack{v\in r(\theta, 1)\\ |\sigma(v)|<n+1}}P_v.
\]

Now suppose that $\omega = (\theta,0)$. We may apply the
conclusion of the preceding paragraph to $(\theta,1)$ for
calculate
\[
Q_\omega
  = Q_{(\theta, 0)}
  = Q_\theta -Q_{(\theta, 1)}
  = Q'_{(\theta, 0)} + \sum_{\substack{v\in r(\theta)\setminus r(\theta, 1)\\ |\sigma(v)|<n+1}}P_v
  =Q'_{(\theta, 0)}+\sum_{\substack{v\in r(\theta, 0)\\ |\sigma(v)|<n+1}}P_v.
\]

Finally, suppose that $\omega = (0^n,1)$. Then we may apply the
conclusion of the preceding paragraph to each $Q'_{(\theta,1)}$
where $\theta \in \{0,1\}^n \setminus \{0^n\}$ to calculate
\begin{flalign*}
&&Q_\omega
  &= Q_{(0^n,1)} &\\
  &&&=Q_{e_{n+1}}-\sum_{\theta \in \{0,1\}^n\setminus\{0^n\}}Q_{(\theta, 1)}&\\
  &&&=\sum_{\substack{\delta \in \{0,1\}^{n+1}\\ \delta_{n+1}=1}}Q'_{\delta}
      +\sum_{\substack{v\in r(e_{n+1})\\ |\sigma(v)|<n+1}}P_v
      -\sum_{\theta \in \{0,1\}^n\setminus\{0^n\}}
      \bigg(Q'_{(\theta, 1)}+\sum_{\substack{v\in r(\theta, 1)\\ |\sigma(v)|<n+1}}P_v \bigg)&\\
  &&&=\sum_{\theta \in \{0,1\}^n}Q'_{(\theta, 1)}
      +\sum_{\substack{v\in r(e_{n+1})\\ |\sigma(v)|<n+1}}P_v
      -\sum_{\theta \in \{0,1\}^n\setminus\{0^n\}}Q'_{(\theta, 1)}
      -\sum_{\substack{\theta \in \{0,1\}^n\setminus\{0^n\}\\ v\in r(\theta, 1)\\ |\sigma(v)|<n+1}}P_v&\\
  &&&=Q'_{(0^n,1)}+\sum_{\substack{v\in r(0^n,1)\\
  |\sigma(v)|<n+1}}P_v &\qedhere
\end{flalign*}
\end{proof}

\begin{corollary}\label{cor:cond(4)check}
For $\omega\in\{0,1\}^n\setminus\{0^n\}$ with $|r(\omega)|<\infty$,
we have
\[
\prod_{\omega_i=1}Q_{e_i}\prod_{\omega_j=0}(1-Q_{e_j})
=\sum_{v\in r(\omega)}P_v.
\]
\end{corollary}
\begin{proof}
Take $\omega\in\{0,1\}^n\setminus\{0^n\}$ with
$|r(\omega)|<\infty$. Then $\omega\notin \Delta$.
Hence
Lemma~\ref{lem:r'empty} and the definition of $r'(\omega)$ imply
that $|\sigma(v)|<|\omega|$ for all $v\in r(\omega)$, and by
definition, $Q'_\omega = 0$.
Thus the conclusion follows from Lemma~\ref{lem:Q_o}.
\end{proof}

\begin{lemma}\label{lem:SeSe*}
For $e_n\in\G^1$,
we have
\[
S_{e_n}S_{e_n}^*=U_{s(e_n)}\bigg(
\sum_{x\in \Xset{n}}t_{\edge(n,x)}t_{\edge(n,x)}^*\bigg)U_{s(e_n)}^*.
\]
\end{lemma}
\begin{proof}
Lemma~\ref{lem:uxux*uyuy*} shows that the $U_x$, $x \in
\Xset{n}$ have mutually orthogonal range projections, and the
result then follows from the definition of $S_{e_n}$.
\end{proof}

\begin{lemma}\label{lem:SeSe*2}
For each $v\in G^0$,
\[
\bigg\{\sum_{x\in \Xset{n}}t_{\edge(n,x)}t_{\edge(n,x)}^* : n \in \N, s(e_n)=v\bigg\}
\]
is a collection of pairwise orthogonal projections dominated by
$q_v$. Moreover, $G^0_{\rg} \subset E^0_{\rg}$, and if $v\in
G^0_{\rg}$ then
\[
\sum_{\{n \in \N \,:\, s(e_n) = v\}} \Big(\sum_{x\in \Xset{n}}t_{\edge(n,x)}t_{\edge(n,x)}^*\Big) = q_v.
\]
\end{lemma}
\begin{proof}
Proposition~\ref{prp:reg verts} shows that
$G^0_{\rg} \subset E^0_\rg$
and both equations then follow from the Cuntz-Krieger relations in $C^*(E)$.
\end{proof}

\begin{proposition}\label{prop:ELGfam}
The collection $\{P_v : v\in G^0\}$ and $\{S_{e_n} :
e_n\in\G^1\}$ is an Exel-Laca $\G$-family in $C^*(E)$.
\end{proposition}
\begin{proof}
By Lemma \ref{lem:P_vQ'_o}~(1) and Lemma
\ref{lem:Q_e}~(1)~and~(2), the collection $\{P_v : v\in G^0\}$
and $\{Q_{e_n} : e_n\in\G^1\}$ satisfies the conditions (1),
(2), and (3) of Definition~\ref{dfn:Cond(EL)}.  It follows from \cite[Corollary~2.18]{KMST} that to establish the $\{P_v\}$ and the $\{Q_{e_n}\}$ satisfy Condition~(EL), it suffices to verify Condition~(4) of Definition~\ref{dfn:Cond(EL)} when $\lambda \cup \mu = \{e_1, \dots, e_n\}$ for some $n$, and this follows from Corollary~\ref{cor:cond(4)check}. The conditions (2) and (3) in Definition
\ref{dfn:EL-G-fam} and the fact that the elements of $\{S_{e_n}
: e_n\in\G^1\}$ have mutually orthogonal ranges follow from
Lemma~\ref{lem:SeSe*} and Lemma~\ref{lem:SeSe*2}.
\end{proof}

\begin{proposition}\label{prop:phi}
There is a strongly continuous action $\beta$ of $\T$ on
$C^*(E)$ satisfying
\begin{itemize}
\item $\beta_z(q_x)=q_x$ for $x\in E^0$,
\item $\beta_z(t_{\overline{x}})=t_{\overline{x}}$ for
  $x\in W_+\sqcup\Gamma_+$, and
\item $\beta_z(t_{\edge(n,x)})=zt_{\edge(n,x)}$ for
    $e_n\in\G^1, x\in \Xset{n}$.
\end{itemize}
Moreover, there is an injective homomorphism $\phi\colon
  C^*(\G) \to C^*(E)$ such that $\phi(p_v)=P_v$ and
  $\phi(s_e)=S_e$, and $\phi$ is equivariant for $\beta$
  and the gauge action on $C^*(E)$.
\end{proposition}
\begin{proof}
The existence of $\beta$ follows from a standard argument using
the universal property of $C^*(E)$.

The first statement of \cite[Corollary~3.5]{KMST} implies that
$C^*(\G)$ is universal for Exel-Laca $\G$-families. Hence there
is a homomorphism $\phi : C^*(\G) \to C^*(E)$ such that
$\phi(p_v)=P_v$ and $\phi(s_e)=S_e$. To prove that $\phi$ is
injective, we first show that $\phi$ is equivariant for $\beta$
and the gauge action on $C^*(E)$, and then apply the
gauge-invariant uniqueness theorem for ultragraphs as stated in
\cite[Corollary~3.5]{KMST}.

It suffices to show that $\beta_z(P_v)=P_v$ and
$\beta_z(S_e)=zS_e$ for $v\in G^0$, $e\in \G^1$ and $z \in \T$.
Each $\beta_z$ fixes $t_\alpha$ for every $\alpha \in F^*$, and
hence fixes the partial isometries $U_x$ of
Definition~\ref{dfn:Ux}. Hence $\beta$ has the desired
properties by definition of the $S_e$ and $P_v$.
\end{proof}

The defining properties of the homomorphism $\phi : C^*(\G) \to C^*(E)$ of the preceding proposition imply that
\[
\phi(p_{r(e_n)}) = \phi(s^*_{e_n} s_{e_n}) =
S_{e_n}^* S_{e_n} = Q_{e_n}
\]
for all $n$, so for all $\omega \in \{0,1\}^n \setminus \{0^n\}$ we have
\[
\phi(p_{r(\omega)})
  = \phi\Big(\prod_{\omega_i = 1} p_{r(e_i)} \prod_{\omega_j = 0} (1 - p_{r(e_j)})\Big)
  = \prod_{\omega_i = 1} Q_{e_i} \prod_{\omega_j = 0} (1 - Q_{e_j})
  = Q_{\omega}.
\]

The following proposition shows that the sets $r'(\omega)$ of
the preceding section and the projections $Q'_\omega$ discussed
in this section satisfy a similar relationship (also compare
Lemma~\ref{lem:r(e_n) and X_n} with Lemma~\ref{lem:Q_e=}, and
Remark~\ref{rem:r'=r'} with Lemma~\ref{lem:Q'}).

\begin{proposition}\label{prop:r'(o)}
For $\omega \in \bigsqcup_{n=1}^\infty (\{0,1\}^n \setminus
\{0^n\})$, the set $r'(\omega)$ is in $\G^0$, and we have
$\phi(p_{r'(\omega)}) = Q'_{\omega}$.
\end{proposition}
\begin{proof}
The set $r'(\omega)$ belongs to $\G^0$ by the definitions of
$r'(\omega)$ and the algebra $\G^0$. That $\phi(p_{r'(\omega)})
= Q'_{\omega}$ follows from the definition of $r'(\omega)$ and
Lemma~\ref{lem:Q_o}.
\end{proof}

We next determine the image of the injection $\phi$ of
Proposition~\ref{prop:phi}.

\begin{lemma}\label{lem:UU^*inImage}
For all $x\in E^0$,
we have $U_xU_x^*\in\phi(C^*(\G))$.
\end{lemma}
\begin{proof}
For $x = v\in G^0$, we have $U_vU_v^*=P_v\in\phi(C^*(\G))$. For
$x = \omega\in \Delta$, we have
\[
U_\omega U_\omega^*
=\prod_{\omega_i=1}Q_{e_i}\prod_{\omega_j=0}(1-Q_{e_j})
-\sum_{\substack{v\in r(\omega)\\ |\sigma(v)|<|\omega|}}P_v \in \phi(C^*(\G))
\]
by Lemma \ref{lem:Q_o}.
\end{proof}

\begin{lemma}\label{lem:tg}
Let $\alpha \in E^*$, and suppose $s_E(\alpha)\in
W_0\sqcup\Gamma_0$. Let
\[
\alpha=
g_0 \cdot \edge(n_1,x_1)\cdot g_1\cdot
\edge(n_2,x_2)\cdot g_{2}\cdots \edge(n_k,x_k)\cdot g_{k}
\]
be the unique expression for $\alpha$ such that each
$e_{n_i}\in\G^1$, $x_i\in \Xset{n_i}$, and $g_i\in F^*$ as in
Lemma~\ref{lem:unique expression}. Then
\[
t_\alpha=S_{e_{n_1}}S_{e_{n_2}}\cdots S_{e_{n_k}}U_{r_E(\alpha)}.
\]
\end{lemma}
\begin{proof}
The proof proceeds by induction on $k$. When $k=0$, the path
$\alpha=g_0$ belongs to $F^*$ with $s_E(\alpha)\in
W_0\sqcup\Gamma_0$. By Definition~\ref{dfn:f-paths}, we have
$\alpha=f_{r_E(\alpha)}$. Hence $t_\alpha=U_{r_E(\alpha)}$.

Suppose as an inductive hypothesis that the result holds for
$k-1$, and fix
\[
\alpha=g_0\cdot \edge(n_1,x_1)\cdot g_1\cdot
\edge(n_2,x_2)\cdot g_{2}\cdots \edge(n_k,x_k)\cdot g_{k}\in E^*.
\]
Let $\alpha'=g_0\cdot \edge(n_1,x_1)\cdot g_1\cdot \edge(n_2,x_2)\cdot
g_{2}\cdots \edge(n_{k-1},x_{k-1})\cdot g_{k-1}$. Then
$\alpha=\alpha'\cdot\edge(n_{k},x_{k})\cdot g_{k}$.
By the inductive hypothesis,
\[
t_\alpha=t_{\alpha'}t_{\edge(n_{k},x_{k})}t_{g_{k}}
=S_{e_{n_1}}S_{e_{n_2}}\cdots S_{e_{n_{k-1}}}U_{r_E(\alpha')}
t_{\edge(n_{k},x_{k})}t_{g_{k}}.
\]
The path $f_{x_{k}}g_{k}$ satisfies
$s_E(f_{x_{k}}g_{k})=s_E(f_{x_{k}})\in W_0\sqcup\Gamma_0$ and
$r_E(f_{x_{k}}g_{k})=r_E(g_{k})=r_E(\alpha)$. By
Definition~\ref{dfn:f-paths} $f_{r_E(\alpha)}=f_{x_{k}}g_{k}$.
Hence $U_{r_E(\alpha)}=U_{x_{k}}t_{g_{k}}$, and
Lemma~\ref{lem:uxux*uyuy*} implies
\[
S_{e_{n_{k}}}U_{r_E(\alpha)}
=\bigg(U_{s(e_{n_{k}})}\sum_{x\in \Xset{n_{k}}}t_{\edge(n_{k},x)}U_{x}^*\bigg)
\big(U_{x_{k}}t_{g_{k}}\big)
=U_{s(e_{n_{k}})}t_{\edge(n_{k},x_{k})}t_{g_{k}}.
\]
Since $r_E(\alpha')=s_E(\edge(n_{k},x_{k}))=s(e_{n_{k}})$,
\[
t_\alpha
=S_{e_{n_1}}S_{e_{n_2}}\cdots S_{e_{n_{k-1}}}S_{e_{n_{k}}}U_{r_E(\alpha)}.
\qedhere
\]
\end{proof}

The sum $\sum_{x\in
W_0\sqcup\Gamma_0}q_x$ converges strictly to a projection $Q
\in \mathcal{M}(C^*(E))$ such that
\[
Q t_\alpha t^*_\beta = \begin{cases}
t_\alpha t^*_\beta &\text{ if $s(\alpha) \in W_0\sqcup\Gamma_0$}\\
0 &\text{ otherwise}
\end{cases}
\]
(see \cite[Lemma~2.10]{Raeburn2005} or
\cite[Lemma~2.1.13]{Tomforde2006} for details). We then have
\begin{align*}
QC^*(E)Q=\cspa\big\{t_{\alpha}t_{\alpha'}^*\in C^*(E):
\alpha,\alpha'\in E^* & \text{ with  $r_E(\alpha)=r_E(\alpha')$} \\
& \text{ and $s_E(\alpha),s_E(\alpha')\in W_0\sqcup\Gamma_0$}\big\}.
\end{align*}

\begin{proposition}\label{prop:Image=QC^*(E)Q}
We have $\phi(C^*(\G)) = QC^*(E)Q$.
\end{proposition}
\begin{proof}
Let $x \in E^0$. Since $s_E(f_x) \in W_0\sqcup\Gamma_0$, we
have $U_x = Q U_x$. This and the definitions of $\{P_v : v\in
G^0\}$ and $\{S_{e_n} : e_n\in\G^1\}$ imply that each $P_v$ and
each $S_{e_n}$ is in $QC^*(E)Q$. Hence $\phi(C^*(\G)) \subset
QC^*(E)Q$. Let $\alpha,\alpha'\in E^*$ such that
$r_E(\alpha)=r_E(\alpha')=x\in E^0$ and
$s_E(\alpha),s_E(\alpha')\in W_0\sqcup\Gamma_0$. By Lemma
\ref{lem:tg},
\[
t_{\alpha}t_{\alpha'}^*=S_{e_{n_1}}S_{e_{n_2}}\cdots S_{e_{n_k}}U_{x}U_{x}^*
S_{e_{m_l}}^*\cdots S_{e_{m_2}}^*S_{e_{m_1}}^*
\]
for some $e_{n_i},e_{m_j} \in \G^1$.
Lemma~\ref{lem:UU^*inImage} therefore implies that
$t_{\alpha}t_{\alpha'}^*\in
\phi(C^*(\G))$. Hence $QC^*(E)Q \subset \phi(C^*(\G))$.
\end{proof}

\begin{lemma}\label{lem:fullproj}
The projection $Q$ is full.
\end{lemma}
\begin{proof}
For $x \in E^0$, we have $U_xU_x^* \in QC^*(E)Q$. Hence
$q_x=U_x^*U_x$ is in the ideal generated by $QC^*(E)Q$. Since
the ideal generated by $\{q_x:x\in E^0\}$ is $C^*(E)$, the
ideal generated by $QC^*(E)Q$ is also $C^*(E)$.
\end{proof}

\begin{theorem}\label{thm:fullcorner}
The homomorphism $\phi$ of Proposition~\ref{prop:phi} is an
isomorphism from $C^*(\G)$ to the full corner $QC^*(E)Q$.
Consequently $C^*(\G)$ and $C^*(E)$ are Morita equivalent.
\end{theorem}
\begin{proof}
This follows from Proposition \ref{prop:phi},
Proposition \ref{prop:Image=QC^*(E)Q},
and Lemma \ref{lem:fullproj}.
\end{proof}

\begin{theorem}
The three classes of graph algebras, of Exel-Laca algebras, and
of ultragraph algebras coincide up to Morita equivalence.
\end{theorem}
\begin{proof}
By \cite[Theorem~4.5 and Remark~4.6]{Tom}, every Exel-Laca
algebra is isomorphic to an ultragraph algebra. Moreover, by
\cite[Theorem~4.5 and Proposition~6.6]{Tom}, every ultragraph
algebra is isomorphic to a full corner of an Exel-Laca algebra.

Proposition~3.1 of \cite{Tom} implies that every graph
$C^*$-algebra is isomorphic to an ultragraph algebra. Finally,
Theorem~\ref{thm:fullcorner} implies that every ultragraph
algebra is Morita equivalent to a graph algebra.
\end{proof}

\begin{remark} \label{rmk:E-L-full-corn}
Note that Theorem~\ref{thm:fullcorner} also shows how to
realize an Exel-Laca algebra as the full corner of a graph
algebra.  If $\mathcal{A}$ is a countably indexed
$\{0,1\}$-matrix with no zero rows, let $\G_A$ be the
ultragraph of \cite[Definition~2.5]{Tom}, which has $A$ as its
edge matrix.  It follows from \cite[Theorem~4.5]{Tom} that the
Exel-Laca algebra $\mathcal{O}_A$ is isomorphic to $C^*(\G_A)$.
If we let $E$ be a graph constructed from $\G_A$ as in
Section~\ref{graph-sec}, then $\mathcal{O}_A$ is isomorphic to
a full corner of $C^*(E)$. It is noteworthy that it seems very
difficult to see how to construct the graph $E$ directly from
the infinite matrix $A$ without at least implicit reference to
the ultragraph $\G_A$.
\end{remark}

\begin{remark}\label{rmk:ultragraph=graph}
With the notation as above, it is straightforward to see that
the following conditions are equivalent:
\begin{enumerate}
\rom
\item The homomorphism $\phi$ of Proposition~\ref{prop:phi}
  is surjective.
\item
The projection $Q$ is the unit of $\mathcal{M} C^*(E)$.
\item
$W_+\sqcup \Gamma_+ = \emptyset$
\item
$\Delta=\emptyset$.
\item
For all $e\in \G^1$, $|r(e)|<\infty$.
\end{enumerate}
In this case, the graph $E=(E^0,E^1,r_E,s_E)$ is obtained as
$E^0=G^0$, $E^1=\{(e,x) : e\in\G^1, x\in r(e)\}$,
$s_E(e,x)=s(e)$ and $r_E(e,x)=x$. In other words, the graph $E$
is obtained from the ultragraph $\G$ by changing each ultraedge
$e\in \G^1$ to a set of (ordinary) edges $\{(e,x) : x\in
r(e)\}$.
\end{remark}

As a consequence of Theorem~\ref{thm:fullcorner}, we obtain the
following characterization of real rank zero for ultragraph
algebras.

\begin{proposition}
Let $\G$ be an ultragraph.
Then $C^*(\G)$ has real rank zero if and only
if $\G$ satisfies Condition~(K).
\end{proposition}
\begin{proof}
Let $E$ be a graph constructed from $\G$ as in
Section~\ref{graph-sec}. By Theorem~\ref{thm:fullcorner},
$C^*(E)$ is Morita equivalent to $C^*(\G)$. Hence
\cite[Theorem~3.8]{BroPed} implies that $C^*(\G)$ has real rank
zero if and only if $C^*(E)$ has real rank zero. By
\cite[Theorem~3.5]{Jeong}, $C^*(E)$ has real rank zero if and
only if $E$ satisfies Condition~(K). By
Proposition~\ref{prop:Cond(K)}, $E$ satisfies Condition~(K) if
and only if $\G$ satisfies Condition~(K).
\end{proof}

\section{Gauge-invariant ideals} \label{ideal-sec}

We continue in this section with a fixed ultragraph $\G$, and
let $E$ be a graph constructed from $\G$ as in
Section~\ref{graph-sec}. We let $Q \in \mathcal{M}(C^*(E))$ and
$\phi : C^*(\G) \to Q C^*(E) Q$ be as in
Theorem~\ref{thm:fullcorner}.

By Theorem~\ref{thm:fullcorner}, the homomorphism $\phi$
induces a bijection from the set of ideals of $C^*(\G)$ to the
set of ideals of $C^*(E)$. We will show in Proposition~\ref{prop:action} that
this bijection restricts to a bijection between gauge-invariant
ideals of $C^*(\G)$ and gauge-invariant ideals of $C^*(E)$.

Let $\beta$ be the action of $\T$ on $C^*(E)$ constructed in
the proof of Proposition \ref{prop:phi}. Specifically,
$\beta_z(q_x)=q_x$ for $x\in E^0$,
$\beta_z(t_{\overline{x}})=t_{\overline{x}}$ for $x\in
W_+\sqcup\Gamma_+$, and
$\beta_z(t_{\edge(n,x)})=zt_{\edge(n,x)}$ for $e_n\in\G^1, x\in
\Xset{n}$. Let $\alpha\in E^*$ and let
\[
\alpha=g_0\cdot (n_1,x_1)\cdot g_1\cdot (n_2,x_2)\cdot g_{2}\cdots
(n_k,x_k)\cdot g_{k}
\]
be the unique expression for $\alpha$ where each $n_i\in\N$,
$x_i\in \Xset{n_i}$ and $g_i\in F^*$ as in
Lemma~\ref{lem:unique expression}. We define
$m(\alpha)=\max\{n_1,\ldots,n_k\}$ and $l(\alpha)=k$. Then one
can verify that $\beta_z(t_{\alpha})=z^{l(\alpha)}t_{\alpha}$.
It follows (see, for example, the argument of
\cite[Corollary~3.3]{Raeburn2005}) that the fixed point algebra
$C^*(E)^{\beta}$ of the action $\beta$ satisfies
\[
C^*(E)^\beta = \cspa\{t_\alpha t_{\alpha'}^* :
\text{$\alpha,\alpha'\in E^*$ with $l(\alpha)=l(\alpha')$}\}.
\]
We define
\[
C^*(E)^{\circ}:=\cspa\{t_\alpha t_\alpha^* \in C^*(E) : \alpha \in E^*\} \subset C^*(E)^\beta.
\]
Then $C^*(E)^\circ$ is an abelian \Csa of $C^*(E)$.

\begin{lemma}\label{lem:beta-inv1}
For $k,n\in\N$ and $x\in E^0$, define
\[
E^*_{k,n,x}:=\{\alpha\in E^* : l(\alpha)=k,\ m(\alpha)\leq n,\ r_E(\alpha)=x\}.
\]
Then $E^*_{k,n,x}$ is finite and
$t_\alpha^*t_{\alpha'}=0$ for $\alpha,\alpha'\in E^*_{k,n,x}$
with $\alpha\neq \alpha'$.
Hence
\[
\A_{k,n,x}:=
\spa\big\{t_{\alpha} t_{\alpha'}^* : \alpha,\alpha'\in E^*_{k,n,x}\big\}
\]
is a \Ca isomorphic to $M_{|E^*_{k,n,x}|}(\C)$.
\end{lemma}
\begin{proof}
For each $x\in E^0$, only finitely many paths $g$ in $F^*$
satisfy $r_E(g)=x$. Hence the set $E^*_{k,n,x}$ is finite. It
is easy to see that the $t_{\alpha} t_{\alpha'}^*$ are matrix
units.
\end{proof}

\begin{lemma}\label{lem:beta-inv2}
For $k,n\in\N$ and $x,y\in E^0$ with $x \neq y$, we have
$\A_{k,n,x}\A_{k,n,y}\subset \A_{k,n,y}$ if there exists a path
in $F^*$ from $x$ to $y$, $\A_{k,n,x}\A_{k,n,y}\subset
\A_{k,n,x}$ if there exists a path in $F^*$ from $y$ to $x$,
and $\A_{k,n,x}\A_{k,n,y}=0$ otherwise.
\end{lemma}
\begin{proof}
We begin by recalling that for any $\alpha$, $\alpha'$,
$\beta$, and $\beta'$ in $E^*$, the Cuntz-Krieger relations
imply that
\begin{equation}\label{eq:monomial product}
t_\alpha t^*_{\alpha'} t_\beta t^*_{\beta'} =
\begin{cases}
t_{\alpha\mu} t^*_{\beta'}  &\text{ if $\beta = \alpha'\mu$ for some $\mu \in E^*$}\\
t_{\alpha} t^*_{\beta'\nu} &\text{ if $\alpha' = \beta\nu$ for some $\nu \in E^*$}\\
0 &\text{ otherwise}.
\end{cases}
\end{equation}
Suppose that $\alpha' \in E^*_{k,n,x}$ and $\beta \in
E^*_{k,n,y}$ satisfy $\beta = \alpha'\mu$ for some $\mu \in
E^*$. We claim that $\mu$ is a path in $F^*$ from $x$ to $y$,
and that $\alpha\mu \in E^*_{k,n,y}$. Since $r(\alpha') = x$ and
$r(\beta) = y$, $\mu$ is a path from $x$ to $y$. Since
$l(\alpha') = l(\beta)$, Lemma~\ref{lem:unique expression} and
the definition of $l$ imply that $\mu \in F^*$. Hence $l(\alpha\mu) = l(\alpha) = k$.
Since every edge in $\alpha\mu$ is an edge in $\alpha$ or an
edge in $\beta$, we also have $m(\alpha\mu) \le
\max\{m(\alpha), m(\beta)\} \le n$. Since $r(\alpha\mu) =
r(\mu) = r(\beta) = y$, it follows that $\alpha\mu \in
E^*_{k,n,y}$ as claimed.

A symmetric argument now shows that if $\alpha' \in
E^*_{k,n,x}$ and $\beta \in E^*_{k,n,y}$ satisfy $\alpha' =
\beta\nu$ for some $\nu \in E^*$, then $\nu$ is a path in $F^*$
from $y$ to $x$ and $\beta'\nu \in E^*_{k,n,y}$.

By Lemma~\ref{lem:paths in F}, $F^*$ contains no return paths.
Since $x \not= y$, it follows that there cannot exist paths
$\mu,\nu \in F^*$ such that $\mu$ is a path from $x$ to $y$ and
$\nu$ is a path from $y$ to $x$. Combining this with the
preceding paragraphs and with~\eqref{eq:monomial product}
proves the result.
\end{proof}

\begin{lemma}\label{lem:center}
Let $\A_0$ and $\A'$ be finite-dimensional \Csas of
a $C^*$-algebra such that $\A_0 \A' \subset \A_0$. Then
$\A:=\A_0+\A'$ is a finite-dimensional \Ca whose center is
contained in the \Ca generated by the center of $\A_0$ and the
center of $\A'$.
\end{lemma}
\begin{proof}
It is clear that $\A:=\A_0+\A'$ is finite dimensional.
It is easy to check that $\A$ is a \Ca
and $\A_0 \subset \A$ is an ideal.
Since $\A_0$ has a unit $p_0$,
$\A$ is the direct sum of $\A_0$
and the \Csa $(1-p_0)\A \subset \A$
where $1$ is the unit of $\A$.
Since $\A=\A_0+\A'$,
the \shom $\A' \ni x \mapsto (1-p_0)x \in (1-p_0)\A$ is
a surjection between finite-dimensional \CA s.
Thus its restriction to the center of $\A'$ is
a surjection onto the center of $(1-p_0)\A$.
This implies that the center of $(1-p_0)\A$
is contained in the \Ca generated
by the center of $\A_0$ and the center of $\A'$
because $p_0$ is in the center of $\A_0$.
Since the center of $\A$ is the direct sum of the center of $\A_0$
and the center of $(1-p_0)\A$,
it is contained in the \Ca generated by the center of $\A_0$ and the
center of $\A'$.
\end{proof}

\begin{lemma}\label{lem:beta-inv3}
Let $k,n\in\N$ and let $\lambda$ be a finite subset $E^0$. Then
$\A_{k,n,\lambda}:=\sum_{x\in \lambda}\A_{k,n,x}$ is a
finite-dimensional \Ca whose center is contained in
$C^*(E)^{\circ}$.
\end{lemma}
\begin{proof}
The proof proceeds by induction on $|\lambda|$. When
$|\lambda|=1$, this follows from Lemma \ref{lem:beta-inv1}.
Suppose the statement holds whenever $|\lambda|=m$. Fix a
finite subset $\lambda \subset E^0$ with $|\lambda|=m+1$. By
Lemma~\ref{lem:paths in F}, $F^*$ contains no return paths, so
there exists $x_0 \in \lambda$ such that there is no path in
$F^*$ from $x_0$ to any other vertex in $\lambda$. Let
$\lambda' = \lambda \setminus \{x_0\}$. Then
Lemma~\ref{lem:beta-inv2} implies that
$\A_{k,n,x_0}\A_{k,n,\lambda\setminus\{x_0\}}\subset
\A_{k,n,x_0}$. Hence $\A_{k,n,\lambda}$ is a finite-dimensional
\Ca whose center is contained in $C^*(E)^{\circ}$ by the
inductive hypothesis applied to $\lambda'$, and
Lemma~\ref{lem:center}.
\end{proof}

\begin{lemma}\label{lem:beta-inv4}
Let $\lambda_1, \lambda_2, \dots$ be an increasing sequence of
finite subsets of $E^0$ such that $\bigcup_{n=1}^\infty
\lambda_n=E^0$. For $n\in\N$ let $\A_n:=\sum_{k=1}^n
\A_{k,n,\lambda_n}$. Then $\A_1, \A_2, \dots$ is an increasing
sequence of finite-dimensional \CA s whose centers are
contained in $C^*(E)^{\circ}$, and the union
$\bigcup_{n=1}^\infty \A_n$ is dense in $C^*(E)^{\beta}$.
\end{lemma}
\begin{proof}
Equation~\ref{eq:monomial product} implies that
$\A_{k',n,\lambda_n}\A_{k,n,\lambda_n}\subset
\A_{k,n,\lambda_n}$ for $k'\leq k$. An argument similar to the
proof of Lemma~\ref{lem:beta-inv3} therefore shows that $\A_n$
is a finite-dimensional \Ca whose center is contained in
$C^*(E)^{\circ}$. By definition, $\{\A_n : n \in \N\}$ is
increasing. The union $\bigcup^\infty_{n=1} \A_n$ is dense in
$C^*(E)^\beta$ because it contains all the spanning elements.
\end{proof}

\begin{lemma}\label{lem:ideal of C*(E)^b}
Every ideal $I$ of $C^*(E)^{\beta}$ is generated as an ideal by
$I\cap C^*(E)^{\circ}$.
\end{lemma}
\begin{proof}
Let $\lambda_1, \lambda_2, \dots$ and $\A_1, \A_2, \dots$ be as
in Lemma~\ref{lem:beta-inv4}. Then $I$ is generated as an ideal
by $\bigcup_{n=1}^\infty I\cap \A_n$. For each $n$, the algebra
$C^*(E)^{\circ}$ contains the center of the finite-dimensional
\Ca $\A_n$, so $I\cap \A_n$ is generated as an ideal by $I\cap
\A_n\cap C^*(E)^{\circ}$. Hence $I$ is generated as an ideal by
$I\cap C^*(E)^{\circ}$.
\end{proof}

\begin{proposition}\label{prop:beta-inv}
Let $I$ be an ideal of $C^*(E)$. Then $I$ is $\beta$-invariant
if and only if $I$ is generated as an ideal by $I\cap
C^*(E)^{\circ}$.
\end{proposition}

To prove the proposition, we first present a well-known
technical lemma, an exact statement of which we have found
difficult to locate in the literature.

\begin{lemma}\label{lem:inv ideal generators}
Let $A$ be a $C^*$-algebra and let $\beta$ be a strongly
continuous action of $\T$ by automorphisms of $A$. An ideal $I$
of $A$ is $\beta$-invariant if and only if it is generated as
an ideal by $I \cap A^\beta$.
\end{lemma}
\begin{proof}
If $I$ is generated as an ideal by $I \cap A^\beta$, then it is
clearly $\beta$-invariant.

Now suppose that $I$ is $\beta$-invariant. Then $I \cap A^\beta
= I^\beta$. Moreover, $\beta$ descends to an action
$\widetilde{\beta}$ of $\T$ on $A/I$, and averaging over
$\beta$ and $\widetilde{\beta}$ gives faithful conditional
expectations $\Phi : A \to A^\beta$ and $\widetilde{\Phi} : A/I
\to (A/I)^{\widetilde{\beta}}$ such that $\widetilde{\Phi}(a +
I) = \Phi(a) + I$.

Let $J \subset I$ be the ideal of $A$ generated by $I^\beta$;
we must show that $J = I$. Fix $a \in J$. Then $a^*a \in J$, so
$\Phi(a^*a) \in J^\beta = I^\beta$. Thus $\widetilde{\Phi}(a^*a
+ I) = \Phi(a^*a) + I = 0_{A/I}$ since $I^\beta \subset I$.
Since $\widetilde{\Phi}$ is faithful, $a^*a + I = 0_{A/I}$, so
the $C^*$-identity implies $a + I = 0_{A/I}$, and $a \in I$.
\end{proof}

\begin{proof}[Proof of Proposition~\ref{prop:beta-inv}]
Since elements in $C^*(E)^{\circ}$ are fixed by $\beta$, if $I$
is generated by $I \cap C^*(E)^{\circ}$, then $I$ is
$\beta$-invariant. Conversely suppose that $I$ is
$\beta$-invariant. Then Lemma~\ref{lem:inv ideal generators}
shows that $I$ is generated as an ideal of $C^*(E)$ by $I\cap
C^*(E)^{\beta}$, and Lemma~\ref{lem:ideal of C*(E)^b} implies
that $I\cap C^*(E)^{\beta}$ is generated as an ideal of
$C^*(E)^\beta$ by $I\cap C^*(E)^{\circ}$.
\end{proof}

The following proposition holds for a general graph $E$.

\begin{proposition}\label{prop:gauge-inv}
An ideal $I$ of $C^*(E)$ is gauge invariant if and only if $I$
is generated by $I\cap C^*(E)^{\circ}$.
\end{proposition}
\begin{proof}
Since $C^*(E)^{\circ}$ is in the fixed point algebra of the
gauge action, the ideal generated by a \Csa of $C^*(E)^{\circ}$
is gauge invariant. Conversely, let $I$ be a gauge-invariant
ideal of $C^*(E)$ and $J$ be the ideal generated by $I\cap
C^*(E)^{\circ}$. Since $I\cap C^*(E)^{\circ}\subset J\subset
I$, we have $J\cap C^*(E)^{\circ}=I\cap C^*(E)^{\circ}$.
Theorem~3.6 of \cite{BHRS} implies that each gauge-invariant
ideal of $C^*(E)$ is uniquely determined by its intersection
with $C^*(E)^\circ$. Since both $I$ and $J$ are gauge
invariant, it follows that $I=J$.
\end{proof}

\begin{proposition}\label{prop:action}
Let $I$ be an ideal of $C^*(\G)$. Then $I$ is invariant under the
gauge action on $C^*(\G)$ if and only if the ideal generated by
$\phi(I)$ is invariant under the gauge action on $C^*(E)$.
\end{proposition}
\begin{proof}
Let $J$ be the ideal generated by $\phi(I)$ in $C^*(E)$. By
Proposition~\ref{prop:phi}, $I$ is invariant under $\gamma$ if
and only if $J$ is invariant under $\beta$. The latter
condition is equivalent to the gauge invariance of $J$ by
Proposition \ref{prop:beta-inv} and Proposition
\ref{prop:gauge-inv}.
\end{proof}

\section{Quotients by gauge-invariant ideals} \label{quotient-sec}

In this section, give a more explicit description of the
bijection between gauge-invariant ideals of $C^*(\G)$ and
gauge-invariant ideals of $C^*(E)$ stated in
Proposition~\ref{prop:action}. To do this we use the
classifications of gauge-invariant ideals in graph algebras
\cite[Theorem~3.6]{BHRS} and in ultragraph algebras
\cite[Theorem~6.12]{KMST}. We also describe quotients of
ultragraph algebras by gauge-invariant ideals as full corners
in graph algebras.

First recall from \cite[Section~6]{KMST} that
an admissible pair for $\G$ consists of a subset $\cH$ of $\G^0$
and a subset $V$ of $G^0$ such that:
\begin{itemize}
\item $\cH$ is an ideal: if $U_1, U_2 \in \cH$ then $U_1
  \cup U_2 \in \cH$, and if $U_1 \in \G^0$, $U_2 \in \cH$
  and $U_1 \subset U_2$, then $U_1 \in \cH$;
\item $\cH$ is hereditary: if $e \in \G^1$ and $\{s(e)\}
 \in \cH$, then $r(e) \in \cH$;
\item $\cH$ is saturated: if $v \in G^0_{\rg}$ and $r(e)
 \in \cH$ for all $e \in s^{-1}(v)$, then $\{v\} \in
 \cH$; and
\item $V \subset \cH^\fin_\infty$, where
\[\textstyle
\cH^\fin_\infty := \{v \in G^0 :
 |s^{-1}(v)| = \infty \text{ and }
 0 < |s^{-1}(v) \cap \{e \in \G^1 : r(e) \notin \cH\}| < \infty\}.
\]
\end{itemize}

Theorem~6.12 of \cite{KMST} shows that there is a bijection $I
\mapsto (\cH_I, V_I)$ between gauge-invariant ideals of
$C^*(\G)$ and admissible pairs for $\G$. Specifically, $\cH_I =
\{U \in \G^0 : p_U \in I\}$ and $V_I = \{v \in
(\cH_I)^\fin_\infty : p_v - \sum_{e \in s^{-1}(v), r(e) \notin
\cH_I} s_e s^*_e \in I\}$.

We must also recall from \cite{BPRS2000} some terminology for a
directed graph $E=(E^0,E^1,r_E,s_E)$. A subset $H$ of $E^0$ is
said to be \emph{hereditary} if $r_E(\alpha) \in H$ whenever
$\alpha \in E^1$ and $s_E(\alpha) \in H$. A hereditary subset
$H$ is said to be \emph{saturated} if $x \in H$ whenever $x \in
E^0_{\rg}$ and $r_E(\alpha) \in H$ for all $\alpha \in
s_E^{-1}(x)$. If $H \subset E^0$ is saturated hereditary, then
we define
\[
H^\fin_\infty := \{v \in E^0 : |s_E^{-1}(v)| = \infty \text{ and }
0 < |s_E^{-1}(v) \cap r_E^{-1}(E^0 \setminus H)| < \infty\}.
\]
Theorem~3.6 of \cite{BHRS} shows that there is a bijection $J
\mapsto (H_J, B_J)$ between gauge-invariant ideals of $C^*(E)$
and pairs $(H,B)$ such that $H \subset E^0$ is saturated
hereditary, and $B \subset H^\fin_\infty$. Specifically,
\[\textstyle
H_J =
\{x \in E^0 : q_x \in J\}
\quad\text{and}\quad
B_J = \{v \in
(H_J)^\fin_\infty : q_v - \sum_{\alpha \in s_E^{-1}(v),
r_E(\alpha) \notin H_J} t_\alpha t_\alpha^* \in J\}.
\]

\begin{definition}
For a saturated hereditary ideal $\cH \subset \G^0$, we define
$\theta(\cH) \subset E^0$ by
\[
\theta(\cH) := \{v \in G^0 : \{v\} \in \cH \}
\cup \{\omega \in \Delta : r'(\omega) \in \cH\}.
\]
\end{definition}

\begin{proposition}\label{prop:ideals Me}
If $I$ is a gauge-invariant ideal of $C^*(\G)$, and $J$ is the ideal of
$C^*(E)$ generated by $\phi(I)$, then $H_J = \theta(\cH_I)$,
$(H_J)^\fin_\infty = (\cH_I)^\fin_\infty$,
and $B_J = V_I$.
\end{proposition}
\begin{proof}
We use the notation established
in Section~\ref{graph-sec} and Section~\ref{corner-sec}.
Let $x \in E^0$.
Since $q_x = U_x^*U_x$,
we have
\[
x \in H_J \iff q_x \in J \iff U_x \in J
 \iff U_xU_x^* \in J.
\]
For $v \in G^0$, we have $U_vU_v^* = \phi(p_v)$. Hence
\[
U_vU_v^* \in J \iff p_v \in I \iff \{v\} \in \cH_I .
\]
Thus $v \in H_J$ if and only if $\{v\} \in \cH_I$. Similarly,
for $\omega \in \Delta$, we have $Q'_\omega = U_\omega
U^*_\omega$ by Definition~\ref{dfn:Q'}, and
Proposition~\ref{prop:r'(o)} implies that $\phi(p_{r'(\omega)})
= Q'_\omega$, so
\[
U_\omega U_\omega^* \in J \iff p_{r'(\omega)} \in I \iff r'(\omega) \in \cH_I .
\]
Thus $\omega \in H_J$ if and only if $r'(\omega) \in \cH_I$.
This shows that $H_J = \theta(\cH_I)$.

Next, we show $(H_J)^\fin_\infty = (\cH_I)^\fin_\infty$. Since
each $\omega \in \Delta$ satisfies $|s_E^{-1}(\omega)| <
\infty$, we have $(H_J)^\fin_\infty \subset G^0$. Fix $v \in
G^0$. We have $s_E^{-1}(v) = \bigsqcup_{s(e_n) = v}
\{\edge(n,x): x\in \Xset{n}\}$. Since each $\Xset{n}$ is
finite, $|s_E^{-1}(v)| = \infty$ if and only if $|s^{-1}(v)| =
\infty$. Lemma~\ref{lem:r(e_n) and X_n} and the conclusion of
the preceding paragraph imply that $r(e_n) \in \cH_I$ if and
only if $\Xset{n} \subset H_J$. Hence
\begin{equation}\label{eq:HJ<->HI}
0 < |s_E^{-1}(v) \cap r_E^{-1}(E^0 \setminus H_J)| < \infty
\ \Longleftrightarrow\
0 < |s^{-1}(v) \cap \{e \in \G^1 : r(e) \notin \cH_I\}| < \infty
\end{equation}
Thus $(H_J)^\fin_\infty = (\cH_I)^\fin_\infty$.

Finally we show $B_J = V_I$. Fix $v \in (H_J)^\fin_\infty =
(\cH_I)^\fin_\infty$. Let $L := \{n : s(e_n) = v, r(e_n) \notin
\cH_I\}$. By~\eqref{eq:HJ<->HI}, we have
\[
\{\alpha \in
s_E^{-1}(v) : r_E(\alpha) \notin H_J\} = \{\edge(n,x) : n \in L, x
\in \Xset{n} \setminus H_J\}.
\]
For $n \in L$ and $x \in \Xset{n} \cap H_J$, we have
$t_{\edge(n,x)}^*t_{\edge(n,x)}= q_x \in J$, and hence
$t_{\edge(n,x)} t_{\edge(n,x)}^* \in J$. Thus
\begin{align*}
q_v - \sum_{\alpha \in s_E^{-1}(v), r_E(\alpha) \notin H_J} t_\alpha t^*_\alpha
  &= q_v - \sum_{\substack{n \in L\\ x \in \Xset{n} \setminus H_J}} t_{\edge(n,x)} t_{\edge(n,x)}^* \\
  &= q_v - \sum_{\substack{n \in L\\ x \in \Xset{n}}} t_{\edge(n,x)} t_{\edge(n,x)}^* + \sum_{\substack{n \in L\\ x \in \Xset{n} \cap H_J}} t_{\edge(n,x)} t_{\edge(n,x)}^*
\end{align*}
belongs to $J$ if and only if
\begin{equation}\label{eq:gap in J}
q_v - \sum_{n \in L, x \in \Xset{n}} t_{\edge(n,x)} t_{\edge(n,x)}^* \in J.
\end{equation}
Moreover, \eqref{eq:gap in J} holds if and only if $p_v -
\sum_{n \in L} s_{e_n} s_{e_n}^* \in I$ because
\[
\phi\Big(p_v - \sum_{n \in L} s_{e_n} s_{e_n}^*\Big)
=P_v -  \sum_{n \in L} S_{e_n} S_{e_n}^*
=U_v\Big(q_v - \sum_{n \in L, x \in \Xset{n}} t_{\edge(n,x)} t_{\edge(n,x)}^*\Big)U_v^*
\]
by Lemma~\ref{lem:SeSe*}. Hence $B_J = V_I$.
\end{proof}

\begin{corollary}
Let $I$ be a gauge-invariant ideal of $C^*(\G)$.
Then the isomorphism $\phi : C^*(\G) \to Q C^*(E) Q$
restricts to an isomorphism of $I$ onto $Q J Q$,
where $J$ is the unique gauge-invariant ideal of $C^*(E)$
such that $H_J = \theta(\cH_I)$ and $B_J = V_I$.
\end{corollary}
\begin{proof}
We have $\phi(I) = Q J Q$
where $J$ is the ideal of $C^*(E)$
generated by $\phi(I)$.
By Proposition~\ref{prop:action} and
Proposition~\ref{prop:ideals Me},
$J$ is the gauge-invariant ideal of $C^*(E)$
such that $H_J = \theta(\cH_I)$ and $B_J = V_I$.
\end{proof}

Using Proposition~\ref{prop:ideals Me} and the results of
\cite{BHRS}, we may now describe quotients of ultragraph
algebras by gauge-invariant ideals as full corners in graph
algebras.

\begin{definition}
Let $I$ be a gauge-invariant ideal of $C^*(\G)$, and let
$\cH_I$, $V_I$, and $\theta(\cH_I)$ be as above. We define a
directed graph $E_I = (E_I^0,E_I^1,r_{E_I},s_{E_I})$ as
follows. The vertex and edge sets are defined by
\begin{align*}
E_I^0 &:= (E^0 \setminus \theta(\cH_I)) \sqcup
\{\wt{x} : x \in (\cH_I)^\fin_\infty\setminus V_I\},\text{ and} \\
E_I^1 &:= r_E^{-1}(E^0 \setminus \theta(\cH_I))
\sqcup \{\wt{\alpha} :
\alpha \in r_E^{-1}((\cH_I)^\fin_\infty\setminus V_I)\subset E^1\}.
\end{align*}
The range and source of $e \in r_E^{-1}(E^0 \setminus
\theta(\cH_I))$ in $E_I$ are the same as those in $E$. For
$\alpha \in r_E^{-1}((\cH_I)^\fin_\infty\setminus V_I)$ we
define $s_{E_I}(\wt{\alpha}) := s_E(\alpha)$, and
$r_{E_I}(\wt{\alpha}) := \wt{x}$ where $x = r_E(\alpha) \in
(\cH_I)^\fin_\infty\setminus V_I$.
\end{definition}

\begin{corollary}\label{cor:Quotients Me}
With the notation above, $C^*(\G)/I$ is isomorphic to a full
corner of $C^*(E_I)$.
\end{corollary}
\begin{proof}
Let $J$ be the ideal of $C^*(E)$ generated by $\phi(I)$. By
Theorem~\ref{thm:fullcorner} the homomorphism $\phi$ induces an
isomorphism $\phi_I \colon C^*(\G)/I \to \overline{Q}
(C^*(E)/J) \overline{Q}$ where $\overline{Q} \in
\mathcal{M}(C^*(E)/J)$ is the image of $Q \in \mathcal{M}
(C^*(E))$ under the extension of the quotient map to multiplier
algebras (see \cite[Corollary~2.51]{TFB}). In particular, the
projection $\overline{Q}$ is full.
By Proposition~\ref{prop:ideals Me} we obtain $H_J =
\theta(\cH_I)$ and $(H_J)^\fin_\infty \setminus B_J
=(\cH_I)^\fin_\infty \setminus V_I$. By
\cite[Corollary~3.5]{BHRS} there is an isomorphism $\psi\colon
C^*(E)/J \to C^*(E_I)$. Let $Q_I \in \mathcal{M} (C^*(E_I))$ be
the image of $\overline{Q}$ under $\psi$. Then $Q_I$ is full,
and $\psi \circ \phi_I \colon C^*(\G)/I \to Q_I C^*(E_I) Q_I$
is the desired isomorphism.
\end{proof}

\end{document}